\documentclass[]{amsart}

\usepackage{hyphenat}
\usepackage{mathtools}
\usepackage{booktabs}
\usepackage{comment}

\usepackage{tikz}
\usetikzlibrary{calc}
\usepackage{pgfplots}
\pgfplotsset{width=6cm}
\pgfplotsset{grid style={dotted}}
\pgfplotsset{
  title style={font=\small},
  tick label style={font=\scriptsize},
  legend style={font=\scriptsize},
}
\pgfplotsset{every axis legend/.append style={
    at={(0.02,0.02)},
    anchor=south west,
    legend cell align=left}}

\renewcommand{\phi}{\varphi}
\newcommand{\trans}{\mathsf{T}}
\newcommand{\isdef}{\mathrel{\mathrel{\mathop:}=}}

\DeclareMathOperator{\Grad}{\nabla}
\DeclareMathOperator{\Div}{div}
\DeclareMathOperator{\Cov}{\mathbb{C}ov}
\DeclareMathOperator{\Mean}{\mathbb{E}}

\DeclareMathOperator{\Img}{img}
\DeclareMathOperator*{\essinf}{ess\,inf}
\DeclareMathOperator*{\esssup}{ess\,sup}
\newcommand{\dif}[2]{\frac{\mathrm{d}^{#2}}{\mathrm{d}{#1}^{#2}}}
\newcommand{\diff}[1]{\,\mathrm{d}#1}
\DeclarePairedDelimiter\groupp{(}{)}%
\DeclarePairedDelimiter\groups{[}{]}%
\DeclarePairedDelimiter\groupb{\{}{\}}%
\DeclarePairedDelimiter\groupa{\langle}{\rangle}%
\DeclarePairedDelimiter\groupc{\lceil}{\rceil}%
\DeclarePairedDelimiter\norms{\lvert}{\rvert}%
\DeclarePairedDelimiter\norm{\lVert}{\rVert}%
\DeclareFontFamily{U}{matha}{\hyphenchar\font45}
\DeclareFontShape{U}{matha}{m}{n}{
      <5> <6> <7> <8> <9> <10> gen * matha
      <10.95> matha10 <12> <14.4> <17.28> <20.74> <24.88> matha12
      }{}
\DeclareSymbolFont{matha}{U}{matha}{m}{n}
\DeclareFontSubstitution{U}{matha}{m}{n}
\DeclareFontFamily{U}{mathx}{\hyphenchar\font45}
\DeclareFontShape{U}{mathx}{m}{n}{
      <5> <6> <7> <8> <9> <10>
      <10.95> <12> <14.4> <17.28> <20.74> <24.88>
      mathx10
      }{}
\DeclareSymbolFont{mathx}{U}{mathx}{m}{n}
\DeclareFontSubstitution{U}{mathx}{m}{n}
\DeclareMathDelimiter{\ltvert}{\mathopen}{matha}{"7E}{mathx}{"17}
\DeclareMathDelimiter{\rtvert}{\mathclose}{matha}{"7E}{mathx}{"17}
\DeclarePairedDelimiter\normt{\ltvert}{\rtvert}%
\newcommand{\Abfm}{\mathbf{A}}
\newcommand{\Bbfm}{\mathbf{B}}

\newcommand{\Fbfm}{\mathbf{F}}

\newcommand{\Ibfm}{\mathbf{I}}

\newcommand{\Mbfm}{\mathbf{M}}

\newcommand{\Vbfm}{\mathbf{V}}

\newcommand{\nbfm}{\mathbf{n}}

\newcommand{\ubfm}{\mathbf{u}}
\newcommand{\vbfm}{\mathbf{v}}
\newcommand{\wbfm}{\mathbf{w}}
\newcommand{\xbfm}{\mathbf{x}}
\newcommand{\ybfm}{\mathbf{y}}

\newcommand{\Nbbb}{\mathbb{N}}

\newcommand{\Pbbb}{\mathbb{P}}

\newcommand{\Rbbb}{\mathbb{R}}

\newcommand{\Bcal}{\mathcal{B}}
\newcommand{\Ccal}{\mathcal{C}}

\newcommand{\Fcal}{\mathcal{F}}

\newcommand{\Mcal}{\mathcal{M}}

\newcommand{\Tcal}{\mathcal{T}}

\newcommand{\Xcal}{\mathcal{X}}

\newcommand{\alphabfm}{{\boldsymbol{\alpha}}}
\newcommand{\betabfm}{{\boldsymbol{\beta}}}
\newcommand{\gammabfm}{{\boldsymbol{\gamma}}}

\newcommand{\mubfm}{{\boldsymbol{\mu}}}

\newcommand{\psibfm}{{\boldsymbol{\psi}}}

\newcommand{\zerobfm}{{\boldsymbol{0}}}

\usepackage{lipsum}
\usepackage{amsfonts}
\usepackage{graphicx}
\usepackage{epstopdf}

\ifpdf
  \DeclareGraphicsExtensions{.eps,.pdf,.png,.jpg}
\else
  \DeclareGraphicsExtensions{.eps}
\fi

\theoremstyle{plain}             
\newtheorem{theorem}{Theorem}[section]
\newtheorem{lemma}[theorem]{Lemma}

\newtheorem{assumption}[theorem]{Assumption}

\begin{document}

\title[UQ for PDEs with Anisotropic Random Diffusion]{Uncertainty Quantification for PDEs with Anisotropic Random Diffusion}

\author{Helmut Harbrecht\and Michael Peters \and Marc Schmidlin}
\address[Helmut Harbrecht\and Michael Peters\and Marc Schmidlin]{
Departement Mathematik und Informatik, Universit\"at Basel, Spiegelgasse 1, 4051 Basel, Schweiz\\
}
\email{\{helmut.harbrecht,michael.peters,marc.schmidlin\}@unibas.ch}

\maketitle

\begin{abstract}
In this article, we consider elliptic diffusion problems 
with an anisotropic random diffusion coefficient. We model 
the notable direction in terms of a random vector field and 
derive regularity results for the solution's dependence on 
the random parameter. It turns out that the decay of the vector 
field's Karhunen-Lo\`eve expansion entirely determines this
regularity. The obtained results allow for sophisticated quadrature 
methods, such as the quasi-Monte Carlo method or the anisotropic 
sparse grid quadrature, in order to approximate quantities of interest, 
like the solution's mean or the variance. Numerical examples in three
spatial dimensions are provided to supplement the presented theory.
\end{abstract}


%
\section{Introduction}
Many phenomena in science and engineering are modelled as
boundary value problems for an unknown function.
Because, in general, the computation of an exact solution
of such a boundary value problem is infeasible,
it is necessary to use numerical schemes that yield
approximations of the solution.
When using numerical methods,
such as finite elements or finite differences,
the behaviour of these numerical simulations is generally well understood
for input data, such as boundary values or coefficients, that are given exactly.
However, for many applications,
the input data are not known exactly and can be thought of being subject to uncertainty,
for example when they are based on measurements.
This then implies that the solution of the boundary value problem is also subject to uncertainty.

Specifically, let us consider the second order diffusion problem
{where uncertainty} regarding the diffusion coefficient $\Abfm$
has been accounted for by considering it as a random matrix field
over a given probability space $\groupp{\Omega, \Fcal, \Pbbb}$, ie.\
\begin{equation*}
  \text{for almost every $\omega \in \Omega$: }
  \left\{
    \begin{alignedat}{2}
      - \Div_\xbfm \groupp[\big]{\Abfm(\omega) \Grad_\xbfm u(\omega)} &= f
      &\quad &\text{in }D , \\
      u(\omega) &= 0 &\quad &\text{on }\Gamma_D , \\
      \groupa[\big]{\Abfm(\omega) \Grad_\xbfm u(\omega), \nbfm} &= g &\quad &\text{on }\Gamma_N ,
    \end{alignedat}
  \right.
\end{equation*}
on a domain $D$ given a mixed boundary condition, where
the function $f$ describes the known source and
the function $g$ the conormal derivative at the Neumann boundary.

We note that the numerical treatment of isotropic random diffusion 
coefficients, ie.\ $\Abfm(\omega) = a(\omega) \Ibfm$, has already 
been considered, for example, in \cite{BC02,BNT,CDS,FST,MK}.
However, since the simulations in applications may require anisotropic 
diffusion, we subsequently consider anisotropic random diffusion 
coefficients that fit the following description: The anisotropic random
diffusion coefficient has, at every point in $D$, a notable direction 
regarding diffusion; that is a direction perpendicular to which diffusion 
is isotropic with a global strength $a$ and in which the strength of 
diffusion may be considerably different and also may vary in space.
We can represent this notable direction with its spatially varying 
directional strength as a random vector field $\Vbfm$.

One interpretation for random diffusion coefficients of this type
is to consider diffusion in a medium that is comprised of thin fibres.
The direction of these fibres is then described by the direction of $\Vbfm$
and the strength of diffusion in a fibre is given by $\norm{\Vbfm}_2$.
The diffusion from a fibre to a neighbouring fibre is the diffusion
that is perpendicular to $\Vbfm$ and is thus considered to have strength $a$. 

The application we have in mind here amounts from cardiac 
electrophysiology: The electrical activation of the human heart 
has been an active area of research during the last decades. It has been
known for a long time that the heart exhibits a fibrous structure. 
By now, it is well-understood that these fibres play a major role for 
both, the electrical and the mechanical properties of the heart muscle.
More precisely, the fibres have a very complex but also well-organized 
structure, exhibiting key features that can be identified in all healthy 
subjects, such as a helical distribution with opposite orientations, from 
the endocardium to the epicardium, see eg.\ \cite{RSG07,SCCM+12}, 
in particular, for visualisations of the cardiac fiber structure. However, 
the exact fibre dislocations vary between different individuals and may 
also change over time within a single individual due to pathologies, 
such as infarctions. Then, the fibre structure is perturbed with the 
introduction of high variability areas in the presence of scars. This 
uncertainty in the fibre direction is modeled via the vector field \({\bf V}\). 
In this particular application, the ratio between \(\|{\bf V}\|_2\) and \(a\) 
is typically of the order ten to one, see eg.\ \cite{GK09,SCCM+12}.

Moreover, random diffusion coefficients of this type may also be used
to model the diffusion in a laminar medium,
ie.\ a medium comprised of stacked thin layers,
by choosing the direction of $\Vbfm$ as the normals on the layers.
Then, we have that in a layer the diffusion is isotropic with the strength $a$
and in between layers, that is in the direction given by $\Vbfm$,
the diffusion strength is given by $\norm{\Vbfm}_2$.

The sections hereafter are organised as follows:
Section~2 introduces some basic definitions and notation
followed by the model problem,
which also includes the formula that expresses $\Abfm$
in terms of $\Vbfm$ and $a$.
Then, in Section~3, we reformulate the model problem
into a stochastically parametric and spatially weak formulation,
by using the Karhunen\hyp{}Lo\`eve expansion of the diffusion describing field $\Vbfm$
to arrive at a stochastically parametrised form of the diffusion coefficient $\Abfm$.
This also enables us to conclude the well\hyp{}posedness of the model problem.
Section~4 is dedicated to the regularity of the solution of the model problem
with respect to the random parameter
from the Karhunen\hyp{}Lo\`eve expansion of the diffusion describing field $\Vbfm$.
Numerical examples are then provided in Section~5 as validation.
Lastly, our conclusions are given in Section~6.

\section{Problem formulation}
\subsection{Notation and precursory remarks}
For a given Banach space $\Xcal$ and a complete measure 
space $\Mcal$ with measure $\mu$ the space $L_\mu^p(\Mcal;\Xcal)$ 
for $1 \leq p \leq \infty$ denotes the Bochner space, see \cite{HP57},
which contains all equivalence classes of strongly measurable functions
$v \colon \Mcal \to \Xcal$ with finite norm
\begin{equation*}
  \norm{v}_{L_\mu^p(\Mcal; \Xcal)} \isdef
  \begin{cases}
    \groups[\bigg]{\displaystyle\int_\Mcal \norm[\big]{v(x)}_{\Xcal}^p \diff{\mu(x)}}^{1/p} , & p < \infty , \\
    \displaystyle\esssup_{x \in \Mcal} \norm[\big]{v(x)}_{\Xcal} , & p = \infty .
  \end{cases}
\end{equation*}
A function $v \colon \Mcal \to \Xcal$ is strongly measurable if there exists 
a sequence of countably\hyp{}valued measurable functions $v_n \colon \Mcal \to \Xcal$,
such that for almost every $m \in \Mcal$ we have $\lim_{n \to \infty} v_n(m) = v(m)$.
Note that, for finite measures $\mu$, we also have the usual inclusion
$L_\mu^p(\Mcal; \Xcal) \supset L_\mu^q(\Mcal; \Xcal)$ for $1 \leq p < q \leq \infty$.

When $\Xcal$ is a separable Hilbert space and $\Mcal$ is a separable measure space,
the Bochner space $L_\mu^2(\Mcal; \Xcal)$ is also a separable Hilbert space with the inner product
\begin{equation*}
  \left(u, v\right)_{L_\mu^2(\Mcal; \Xcal)} \isdef \int_\Mcal \groupp[\big]{u(x), v(x)}_{\Xcal} \diff{\mu(x)} 
\end{equation*}
and is isomorphic to the tensor product space $L_\mu^2(\Mcal) \otimes \Xcal$, see \cite{LC}.

Subsequently, we will always equip $\Rbbb^d$ with the norm $\norm{\cdot}_2$
induced by the canonical inner product $\langle\cdot,\cdot\rangle$
and $\Rbbb^{d \times d}$ with the norm $\norm{\cdot}_F$
induced by the Frobenius inner product $\langle\cdot,\cdot\rangle_F$.
Then, for $\vbfm, \wbfm \in \Rbbb^d$, the Cauchy\hyp{}Schwartz inequality gives us
\begin{equation*}
  \norms{\vbfm^\trans \wbfm} = \norms[\big]{\groupa{\vbfm, \wbfm}} \leq \norm{\vbfm}_2 \norm{\wbfm}_2
  \quad\text{and the special case}\quad
  \vbfm^\trans \vbfm = \groupa{\vbfm, \vbfm} = \norm{\vbfm}_2^2 ,
\end{equation*}
and we also have, by straightforward computation, that
\(
  \norm{\vbfm \wbfm^\trans}_F = \norm{\vbfm}_2 \norm{\wbfm}_2 .
\)

\subsection{The model problem}
Let $(\Omega, \Fcal, \Pbbb)$ be a separable, complete probability space.
We consider the following second order diffusion problem with a random diffusion coefficient
\begin{equation}
  \label{eq:smwpoisson}
  \text{for almost every $\omega \in \Omega$: } \left\{
    \begin{alignedat}{2}
      - \Div_\xbfm \groupp[\big]{\Abfm(\omega) \Grad_\xbfm u(\omega)} &= f
      &\quad &\text{in }D , \\
      u(\omega) &= 0 &\quad &\text{on }\Gamma_D , \\
      \groupa[\big]{\Abfm(\omega) \Grad_\xbfm u(\omega), \nbfm} &= g &\quad &\text{on }\Gamma_N ,
    \end{alignedat}
  \right.
\end{equation}
where $D \subset \Rbbb^d$ is a Lipschitz domain with $d \geq 1$
and $\partial D = \overline{\Gamma}_D \cup \overline{\Gamma}_N$
is a disjoint decomposition of the boundary. The function $f \in 
\widetilde{H}^{-1}(D)$ describes the known source and $g \in 
H^{-1/2}(\Gamma_N)$ the known conormal derivative at the 
Neumann boundary. The random matrix field $\Abfm \in 
L_\Pbbb^\infty\groupp[\big]{\Omega; L^\infty(D; \Rbbb^{d \times d})}$ 
is the stochastically and spatially varying diffusion coefficient,
which satisfies the uniform ellipticity condition
\begin{equation}
  \label{eq:Aellipticity}
  a_{\min} \leq \essinf_{\xbfm \in D} \lambda_{\min}\groupp[\big]{\Abfm(\xbfm, \omega)} \leq
  \esssup_{\xbfm \in D} \lambda_{\max}\groupp[\big]{\Abfm(\xbfm, \omega)} \leq a_{\max}
  \quad\text{$\Pbbb$-almost surely}
\end{equation}
for some constants $0 < a_{\min} \leq a_{\max} < \infty$
and is almost surely symmetric almost everywhere. Without loss 
of generality, we assume $a_{\min} \leq 1$ and $a_{\max} \geq 1$.

To be able to capture the two situations arising from the fact
that we might have a pure Neumann boundary value problem,
we introduce the solution space $V$ dependent upon the boundary\hyp{}measure of $\Gamma_D$:
If $\Gamma_D$ has non\hyp{}zero boundary\hyp{}measure,
    then we define
    \begin{equation*}
      V \isdef H_{\Gamma_D}^1(D)
      \isdef \groupb[\big]{v \in H^1(D) : v(\xbfm) = 0 \text{ for all } \xbfm \in \Gamma_D} .
    \end{equation*}
   If $\Gamma_D$ has zero boundary\hyp{}measure,
    ie.\ if we have a pure Neumann boundary value problem,
    then we set
    \[
      V \isdef H_{*}^1(D)
      \isdef \groupb[\big]{v \in H^1(D) : (v,1)_{L^2(D)} = 0}
    \]
       and also require that $f$ and $g$ fulfil the compatibility condition
    \begin{equation*}
      \int_D f(\xbfm) \diff{\xbfm} = - \int_{\Gamma_N} g(\xbfm) \diff{s} .
    \end{equation*}

In both cases, the norm equivalence theorem of Sobolev, see \cite{A},
implies for all \(v\in V\) and some constant {\(0<c_V\leq1\)} that
\begin{equation*}
  c_V \norm{v}_{H^1(D)} \leq \norm{v}_V = \norm[\big]{\Grad_{\xbfm} v}_{L^2(D; \Rbbb^d)} \leq \norm{v}_{H^1(D)} .
\end{equation*}

We will consider anisotropic diffusion coefficients that are of the form
\begin{equation}
  \label{eq:AdmV}
  \Abfm(\xbfm, \omega) \isdef a \Ibfm + \groupp[\Big]{\norm[\big]{\Vbfm(\xbfm, \omega)}_2 - a} \frac{\Vbfm(\xbfm, \omega) \Vbfm^\trans(\xbfm, \omega)}{\Vbfm^\trans(\xbfm, \omega) \Vbfm(\xbfm, \omega)} ,
\end{equation}
where $a \in \Rbbb$ is a given value and
$\Vbfm \in L_\Pbbb^\infty\groupp[\big]{\Omega; L^\infty(D; \Rbbb^d)}$ is a random vector field.
Furthermore, we require that they satisfy $b_{\min} \leq a \leq b_{\max}$ and
\begin{equation}
  \label{eq:Vellipticity}
  b_{\min} \leq \essinf_{\xbfm \in D} \norm[\big]{\Vbfm(\xbfm, \omega)}_2
  \leq \esssup_{\xbfm \in D} \norm[\big]{\Vbfm(\xbfm, \omega)}_2 \leq b_{\max}
  \quad\text{$\Pbbb$-almost surely}
\end{equation}
for some constants $0 < b_{\min} \leq b_{\max} < \infty$. Without 
loss of generality, we assume $b_{\min} \leq 1$ and $b_{\max} \geq 1$.

We note that the field $\Abfm$ accounts for a medium that has 
homogeneous diffusion strength $a$ perpendicular to $\Vbfm$
and has diffusion strength $\norm[\big]{\Vbfm(\xbfm, \omega)}_2$
in the direction of $\Vbfm$. The randomness of the specific direction 
and length of $\Vbfm$ therefore quantifies the uncertainty of this 
notable direction and its diffusion strength.

\begin{lemma}
  A diffusion coefficient of form \eqref{eq:AdmV} is well\hyp{}formed
  and indeed also satisfies the uniform ellipticity condition \eqref{eq:Aellipticity}
  with $b_{\min}$ and $b_{\max}$.
\end{lemma}
\begin{proof}
  For almost every $\omega \in \Omega$ and almost every $\xbfm \in D$ we have
  that $\Abfm(\xbfm, \omega)$ is well\hyp{}formed, because of
  \begin{equation*}
    \Vbfm^\trans(\xbfm, \omega) \Vbfm(\xbfm, \omega) = \norm{\Vbfm(\xbfm, \omega)}_2^2 \geq b_{\min}^2 > 0 ,
  \end{equation*}
  and clearly symmetric. 
  Furthermore, we can choose $\ubfm_2, \ldots, \ubfm_d \in \Rbbb^d$ that are perpendicular to $\Vbfm(\xbfm, \omega)$
  and are linearly independent;
  then, we know that, for all $i = 2, \ldots, d$,
  \begin{equation*}
    \Abfm(\xbfm, \omega) \ubfm_i = a \ubfm_i
\quad\text{and}\quad 
    \Abfm(\xbfm, \omega) \Vbfm(\xbfm, \omega) = \norm[\big]{\Vbfm(\xbfm, \omega)}_2 \Vbfm(\xbfm, \omega) .
  \end{equation*}
  This means that, for almost every $\omega \in \Omega$ and almost every $\xbfm \in D$,
  \begin{align*}
    \lambda_{\min}\groupp[\big]{\Abfm(\xbfm, \omega)} &= \min\groupb[\Big]{a, \norm[\big]{\Vbfm(\xbfm, \omega)}_2} \geq b_{\min},\\
    \lambda_{\max}\groupp[\big]{\Abfm(\xbfm, \omega)} &= \max\groupb[\Big]{a, \norm[\big]{\Vbfm(\xbfm, \omega)}_2} \leq b_{\max}.
  \end{align*}
Therefore, $\Abfm$ satisfies the uniform ellipticity condition \eqref{eq:Aellipticity}
with $b_{\min}$ and $b_{\max}$.
\end{proof}

Thus, we will set $a_{\min} \isdef b_{\min}$ and $a_{\max} \isdef b_{\max}$
and, from here on, solely use $a_{\min}$ and $a_{\max}$.

\section{Problem reformulation}
\subsection{Karhunen-Lo\`eve expansion}
To make the random field, and hence also the diffusion 
coefficient, feasible for numerical computations, we separate 
the spatial variable $\xbfm$ and the stochastic parameter $\omega$
by considering the Karhunen\hyp{}Lo\`eve expansion of $\Vbfm$.
The mean field $\Mean[\Vbfm] \colon \Omega \to \Rbbb^d$
and the matrix\hyp{}valued covariance field $\Cov[\Vbfm] 
\colon D \times D \to \Rbbb^{d \times d}$ are given by
\begin{equation*}
  \Mean[\Vbfm](\xbfm)
  = \int_\Omega \Vbfm(\xbfm, \omega) \diff{\Pbbb(\omega)}
\end{equation*}
and
\begin{equation*}
  \Cov[\Vbfm](\xbfm, \xbfm')
  = \int_\Omega \Vbfm_0(\xbfm, \omega) \Vbfm_0^\trans(\xbfm', \omega) \diff{\Pbbb(\omega)} ,
\end{equation*}
respectively, where
\begin{equation*}
  \Vbfm_0(\xbfm, \omega) \isdef \Vbfm(\xbfm, \omega) - \Mean[\Vbfm](\xbfm) .
\end{equation*}

Let \(\{\lambda_k,\psibfm_k\}_k\) denote the eigenpairs 
corresponding to the Hilbert-Schmidt operator {\(\Ccal\)} that is induced 
from the kernel \(\Cov[\Vbfm](\xbfm, \xbfm')\), ie.
\begin{equation*}
    (\Ccal \ubfm)(\xbfm)\isdef \int_D \Cov[\Vbfm](\xbfm, \xbfm') \ubfm(\xbfm') \diff{\xbfm'} .
\end{equation*}
Then, the Karhunen Lo\`eve expansion of \({\bf V}\) is given by
\begin{equation}\label{eq:KL}
  \Vbfm(\xbfm, \omega)
  = \Mean[\Vbfm](\xbfm) + \sum_{k=1}^{\infty} \sqrt{\lambda_k} \psibfm_k(\xbfm) Y_k(\omega),
\end{equation}
where the uncorrelated and centered random variables 
\(\{Y_k\}_k\) are given according to
\begin{equation*}
  Y_k(\omega) \isdef \frac{1}{\sqrt{\lambda_k}} \int_D 
  \Vbfm_0^\trans(\xbfm, \omega) \psibfm_k(\xbfm) \diff{\xbfm}.
\end{equation*}
We note that the convergence of \eqref{eq:KL} is in the 
$L_\Pbbb^2\groupp[\big]{\Omega; L^2(D; \Rbbb^d)}$\hyp{}norm.
However, we have the following lemma.

\begin{lemma}
  We have $\Img(\Ccal) \subset L^\infty(D; \Rbbb^{d})$.
  This implies that $\psibfm_k \in L^\infty(D; \Rbbb^d)$
  and, as a consequence, also $Y_k(\omega) \in L_\Pbbb^\infty(\Omega)$.
\end{lemma}
\begin{proof}
For almost every $\xbfm \in D$, we know that
  $\Cov[\Vbfm](\xbfm, \cdot) \in L^\infty(D; \Rbbb^{d \times d})$,
  where we clearly have
  \begin{equation*}
    \norm[\big]{\Cov[\Vbfm](\xbfm, \cdot)}_{L^\infty(D; \Rbbb^{d \times d})}
    \leq \norm[\big]{\Cov[\Vbfm]}_{L^\infty(D; L^\infty(D; \Rbbb^{d \times d}))} .
  \end{equation*}
  Thus, we can calculate for almost every $\xbfm \in D$ that
  \begin{align*}
    \norm[\big]{(\Ccal \ubfm)(\xbfm)}_2
    &
    \leq \int_D \norm[\big]{\Cov[\Vbfm](\xbfm, \xbfm') \ubfm(\xbfm')}_2 \diff{\xbfm'} \\
    &\leq \int_D \norm[\big]{\Cov[\Vbfm](\xbfm, \xbfm')}_F \norm[\big]{\ubfm(\xbfm')}_2 \diff{\xbfm'} \\
    &\leq \norm[\big]{\Cov[\Vbfm]}_{L^\infty(D; L^\infty(D; \Rbbb^{d \times d}))}
    \sqrt{\norms{D}}
    \norm{\ubfm}_{L^2(D; \Rbbb^d)} 
  \end{align*}
  and conclude that
  \(
    \norm{\Ccal \ubfm}_{L^\infty(D; \Rbbb^{d})}
    \leq \norm[\big]{\Cov[\Vbfm]}_{L^\infty(D; L^\infty(D; \Rbbb^{d \times d}))}
    \sqrt{\norms{D}} \norm{\ubfm}_{L^2(D; \Rbbb^d)}. 
  \)
\end{proof}

Now, by parametrisation of the $Y_k$ as $y_k$ and replacing the 
$\sqrt{\lambda_k}$ with $\sigma_k$, we may assume, without loss 
of generality, that $y_k \in \groups{{-1}, 1}$, when considering the 
vector field $\Vbfm$ in the parametrised form
\begin{equation}
  \label{eq:KLp}
  \Vbfm(\xbfm, \ybfm)
  = \Mean[\Vbfm](\xbfm) + \sum_{k=1}^{\infty} \sigma_k \psibfm_k(\xbfm) y_k ,
\end{equation}
where $\ybfm = (y_k)_{k \in \Nbbb} \in \square \isdef \groups{{-1}, 1}^{\Nbbb}$.
Consequently, we can also view $\Abfm(\xbfm, \ybfm)$ and $u(\xbfm, \ybfm)$ 
as being parametrised by $\ybfm$ and restate \eqref{eq:smwpoisson} as
\begin{equation}
  \label{eq:spmwpoisson}
  \text{for almost every $\ybfm \in \square$: } \left\{
    \begin{alignedat}{2}
      - \Div_\xbfm \groupp[\big]{\Abfm(\ybfm) \Grad_\xbfm u(\ybfm)} &= f
      &\quad &\text{in }D , \\
      u(\ybfm) &= 0 &\quad &\text{on }\Gamma_D, \\
      \groupa[\big]{\Abfm(\ybfm) \Grad_\xbfm u(\ybfm), \nbfm} &= g &\quad &\text{on } \Gamma_N.
    \end{alignedat}
  \right.
\end{equation}

We now impose some common assumptions,
which make the Karhunen\hyp{}Lo\`eve expansion computationally feasible.
\begin{assumption}
The random variables $(Y_k)_{k \in \Nbbb}$ are independent and uniformly distributed
on $\groups[\big]{{-\sqrt{3}}, \sqrt{3}}$, ie.\ \(\sigma_k=\sqrt{3\lambda_k}\). Moreover, the sequence 
$\gammabfm = \groupp{\gamma_k}_{k \in \Nbbb_0}$, given by 
\begin{equation*}
      \gamma_k \isdef \norm[\big]{\sigma_k \psibfm_k}_{L^\infty(D; \Rbbb^d)} ,
    \end{equation*}
    is at least in $\ell^1(\Nbbb_0)$,
    where we have defined $\psibfm_0 \isdef \Mean[\Vbfm]$ and
    $\sigma_0 \isdef 1$.
\end{assumption}

\begin{lemma}
  The representation \eqref{eq:KLp} {also} converges in
  $L_\Pbbb^\infty\groupp[\big]{\Omega; L^\infty(D; \Rbbb^d)}$.
\end{lemma}
\begin{proof}
  We define
  \begin{equation*}
    \Vbfm^M(\xbfm, \ybfm) \isdef \Mean[\Vbfm](\xbfm) + \sum_{k=1}^{M} \sigma_k \psibfm_k(\xbfm) y_k .
  \end{equation*}
  Since $L_{\Pbbb_\ybfm}^\infty\groupp[\big]{\square; L^\infty(D; \Rbbb^d)}$ is complete,
  it suffices to show that
  $\groupp{\Vbfm^M}_{M \in \Nbbb}$ is a Cauchy sequence
  in $L_{\Pbbb_\ybfm}^\infty\groupp[\big]{\square; L^\infty(D; \Rbbb^d)}$.
  Let $M \leq M'$ be two indices, then we have
  \begin{align*}
    \norm[\big]{\Vbfm^{M'} - \Vbfm^M}_{L_{\Pbbb_\ybfm}^\infty(\square; L^\infty(D; \Rbbb^d))}
    &= \norm[\bigg]{\sum_{k=M+1}^{M'} \sigma_k \psibfm_k y_k}_{L_{\Pbbb_\ybfm}^\infty(\square; L^\infty(D; \Rbbb^d))} \\
    &\leq \sum_{k=M+1}^{M'} \norm[\big]{\sigma_k \psibfm_k}_{L^\infty(D; \Rbbb^d)} \norm[\big]{y_k}_{L_{\Pbbb_\ybfm}^\infty(\square; \Rbbb)} \\
    &\leq \sum_{k=M+1}^{M'} \gamma_k
    \leq \sum_{k=M+1}^{\infty} \gamma_k .
  \end{align*}
  Thus, since $\gammabfm \in \ell^1(\Nbbb_0)$, we know that
\(
    \norm[\big]{\Vbfm^{M'} - \Vbfm^M}_{L_{\Pbbb_\ybfm}^\infty(\square; L^\infty(D; \Rbbb^d))} \xrightarrow{M, M' \to \infty} 0
\)
  and so $\groupp{\Vbfm^M}_{M \in \Nbbb}$  is a Cauchy sequence
  in $L_{\Pbbb_\ybfm}^\infty\groupp[\big]{\square; L^\infty(D; \Rbbb^d)}$.
\end{proof}

\subsection{Spatially weak formulation}
Since we want to pursue a finite element approach in space to approximate the
solution of \eqref{eq:spmwpoisson},
we will need the spatially weak form thereof.

Given almost any $\ybfm \in \square$, we have
\begin{equation*}
  - \Div_\xbfm \groupp[\big]{\Abfm(\xbfm, \ybfm) \Grad_\xbfm u(\xbfm, \ybfm)} = f(\xbfm)
  \quad \text{for all $\xbfm \in D$.}
\end{equation*}
After multiplication with a test function $v \in V$ and
integration over $D$, we arrive at
\begin{equation*}
  - \int_D \Div_\xbfm \groupp[\big]{\Abfm(\xbfm, \ybfm) \Grad_\xbfm u(\xbfm, \ybfm)} v(\xbfm) \diff{\xbfm}
  = \int_D f(\xbfm) v(\xbfm) \diff{\xbfm} .
\end{equation*}
Now, Green's identity implies
\begin{align*}
  - \int_D \Div_\xbfm \groupp[\big]{\Abfm(\xbfm, \ybfm) \Grad_\xbfm u(\xbfm, \ybfm)} v(\xbfm) \diff{\xbfm}
  &= \int_D \groupa[\big]{\Abfm(\xbfm, \ybfm) \Grad_\xbfm u(\xbfm, \ybfm), \Grad_\xbfm v(\xbfm)} \diff{\xbfm} \\
  &\quad {} - \int_{\partial D} \groupa[\big]{\Abfm(\xbfm, \ybfm) \Grad_\xbfm u(\xbfm, \ybfm), \nbfm(\xbfm)} v(\xbfm) \diff{s} ,
\end{align*}
which, because of $v|_{\Gamma_D} = 0$ and
$\groupa[\big]{\Abfm(\ybfm) \Grad_\xbfm u(\ybfm), \nbfm} = g$ on $\Gamma_N$, simplifies to
\begin{align*}
  - \int_D \Div_\xbfm \groupp[\big]{\Abfm(\xbfm, \ybfm) \Grad_\xbfm u(\xbfm, \ybfm)} v(\xbfm) \diff{\xbfm}
  &= \int_D \groupa[\big]{\Abfm(\xbfm, \ybfm) \Grad_\xbfm u(\xbfm, \ybfm), \Grad_\xbfm v(\xbfm)} \diff{\xbfm} \\
  &\quad {} - \int_{\Gamma_N} g(\xbfm) v(\xbfm) \diff{s} .
\end{align*}
We define $\Bcal \colon \square \to \groupp[\big]{V \times V \to \Rbbb}$,
where $\Bcal[\ybfm]$ is a continuous symmetric bilinear form for almost any $\ybfm \in \square$, by
\begin{equation*}
  \Bcal[\ybfm](u, v) \isdef \int_D \groupa[\big]{\Abfm(\xbfm, \ybfm) \Grad_\xbfm u(\xbfm), \Grad_\xbfm v(\xbfm)} \diff{\xbfm}
\end{equation*}
and $\ell \colon V \to \Rbbb$ a continuous linear form by
\begin{equation*}
  \ell(v) \isdef \int_D f(\xbfm) v(\xbfm) \diff{\xbfm} + \int_{\Gamma_N} g(\xbfm) v(\xbfm) \diff{s} .
\end{equation*}
Then, this leads us to the spatially weak formulation
\begin{equation}
  \label{eq:smwpoissonsw}
  \left\{
  \begin{aligned}
    & \text{Find $u \colon \square \to V$ such that} \\
    & \quad \Bcal[\ybfm]\groupp[\big]{u(\ybfm), v} = \ell(v)
    \quad \text{for almost every $\ybfm \in \square$ and all $v \in V$.}
  \end{aligned}
  \right.
\end{equation}
We conclude the following well known stability estimate.
\begin{lemma}\label{lemma:ubound}
  For almost every $\ybfm \in \square$,
  there is a unique solution $u(\ybfm) \in V$ of \eqref{eq:smwpoissonsw}, which fulfils
  \begin{equation*}
    \norm[\big]{u(\ybfm)}_{H^1(D)} 
    \leq \frac{a_{\max}}{a_{\min} c_V^2} \groupp[\Big]{\norm{f}_{\widetilde{H}^{-1}(D)} + \norm{g}_{H^{-1/2}(\Gamma_N)}} .
  \end{equation*}
\end{lemma}


\section{Parametric regularity}
\subsection{Parametric regularity of the diffusion coefficient}
In the following, we will use the notation
\begin{gather*}
  \normt{s} \isdef \norm{s}_{L_{\Pbbb_\ybfm}^\infty(\square; L^\infty(D; \Rbbb))}
  = \esssup_{\ybfm \in \square} \esssup_{\xbfm \in D} \norms[\big]{s(\xbfm, \ybfm)} , \\
  \normt{\vbfm}_d \isdef \norm{\vbfm}_{L_{\Pbbb_\ybfm}^\infty(\square; L^\infty(D; \Rbbb^d))}
  = \esssup_{\ybfm \in \square} \esssup_{\xbfm \in D} \norm[\big]{\vbfm(\xbfm, \ybfm)}_2 , \\
  \normt{\Mbfm}_{d \times d} \isdef \norm{\Mbfm}_{L_{\Pbbb_\ybfm}^\infty(\square; L^\infty(D; \Rbbb^{d \times d}))} 
  = \esssup_{\ybfm \in \square} \esssup_{\xbfm \in D} \norm[\big]{\Mbfm(\xbfm, \ybfm)}_F
\end{gather*}
for $s \in L_{\Pbbb_\ybfm}^\infty\groupp[\big]{\square; L^\infty(D; \Rbbb)}$,
$\vbfm \in L_{\Pbbb_\ybfm}^\infty\groupp[\big]{\square; L^\infty(D; \Rbbb^d)}$ and
$\Mbfm \in L_{\Pbbb_\ybfm}^\infty\groupp[\big]{\square; L^\infty(D; \Rbbb^{d \times d})}$.
We will further make extensive use of the following straightforward result.
\begin{lemma}
  Given $\vbfm, \wbfm \in L_{\Pbbb_\ybfm}^\infty\groupp[\big]{\square; L^\infty(D; \Rbbb^d)}$, we have
  \begin{equation*}
    \normt[\big]{\vbfm^\trans \wbfm}
    \leq \normt{\vbfm}_d \normt{\wbfm}_d
    \quad\text{and}\quad
    \normt[\big]{\vbfm^\trans \vbfm}
    = \normt{\vbfm}_d^2
  \end{equation*}
  as well as
  \begin{equation*}
    \normt[\big]{\vbfm \wbfm^\trans}_{d \times d}
    \leq \normt{\vbfm}_d \normt{\wbfm}_d
    \quad\text{and}\quad
    \normt[\big]{\vbfm \vbfm^\trans}_{d \times d}
    = \normt{\vbfm}_d^2.
  \end{equation*}
\end{lemma}


In this section, we assume that the vector field $\Vbfm$ is
given by a finite rank Karhunen\hyp{}Lo\`eve expansion, ie.\
\begin{equation*}
  \Vbfm(\xbfm, \ybfm)
  = \psibfm_0(\xbfm) + \sum_{k=1}^{M} \sigma_k \psibfm_k(\xbfm) y_k .
\end{equation*}
If necessary this can be attained by appropriate truncation:
\begin{lemma}
  The condition \eqref{eq:Vellipticity} is satisfied
  by any truncation of the Karhunen\hyp{}Lo\`eve expansion
  with a large enough $M$.
\end{lemma}
\begin{proof}
  Recall the definition
  \begin{equation*}
    \Vbfm^M(\xbfm, \ybfm) \isdef \Mean[\Vbfm](\xbfm) + \sum_{k=1}^{M} \sigma_k \psibfm_k(\xbfm) y_k .
  \end{equation*}
  Clearly, for any $M$ we have that
  \begin{align*}
    \normt[\big]{\Vbfm^M}_{d}
    &= \normt[\bigg]{\Mean[\Vbfm] + \sum_{k=1}^{M} \sigma_k \psibfm_k y_k}_{d}
    = \normt[\bigg]{\psibfm_0 + \sum_{k=1}^{M} \sigma_k \psibfm_k y_k}_{d} \\
    &\leq \normt[\big]{\psibfm_0}_{d} + \sum_{k=1}^{M} \normt[\big]{\sigma_k \psibfm_k y_k}_{d} 
    \leq \sum_{k=0}^{M} \gamma_k \leq \sum_{k=0}^{\infty} \gamma_k < \infty ,
  \end{align*}
  since $\gammabfm \in \ell^1(\Nbbb_0)$.
  Then, because of
  \begin{align*}
    \normt[\big]{\Vbfm^M}_{d}
    &= \normt[\bigg]{\Mean[\Vbfm] + \sum_{k=1}^{M} \sigma_k \psibfm_k y_k}_{d} 
    = \normt[\bigg]{\Vbfm - \sum_{k=M+1}^{\infty} \sigma_k \psibfm_k y_k}_{d} \\
    &\geq \normt{\Vbfm}_{d} - \normt[\bigg]{\sum_{k=M+1}^{\infty} \sigma_k \psibfm_k y_k}_{d}
    \geq a_{\min} - \sum_{k=M+1}^{\infty} \gamma_k ,
  \end{align*}
  for any $M$ that fulfils $\sum_{k=M+1}^{\infty} \gamma_k < a_{\min}$,
  we can find constants 
  with which $\Vbfm^M$ satisfies the condition \eqref{eq:Vellipticity}.
  Since $\gammabfm \in \ell^1(\Nbbb_0)$ implies that
\(
    \sum_{k=M+1}^{\infty} \gamma_k \xrightarrow{M \to \infty} 0 
\)
, we see that
  $\sum_{k=M+1}^{\infty} \gamma_k < a_{\min}$ is fulfilled for sufficiently large $M$.
\end{proof}

We shall now provide regularity estimates for the different terms in \eqref{eq:AdmV}.
\begin{lemma}\label{lemma:Bybounds}
  Let $\Bbfm$ be defined as
  \(
    \Bbfm(\xbfm, \ybfm) \isdef \Vbfm(\xbfm, \ybfm) \Vbfm^\trans(\xbfm, \ybfm)
  \).
  Then, we have for all $\alphabfm \in \Nbbb_0^M$ that
  \begin{equation*}
    \normt[\big]{\partial_\ybfm^\alphabfm \Bbfm}_{d \times d}
    \leq 2 a_{\max}^2 \gammabfm^\alphabfm .
  \end{equation*}
\end{lemma}
\begin{proof}
  More verbosely, $\Bbfm$ is given by
  \begin{equation*}
    \Bbfm(\xbfm, \ybfm)
    = \groupp[\bigg]{\psibfm_0(\xbfm)  + \sum_{k=1}^{M} \sigma_k \psibfm_k(\xbfm) y_k}
    \groupp[\bigg]{\psibfm_0(\xbfm)  + \sum_{k=1}^{M} \sigma_k \psibfm_k(\xbfm) y_k}^\trans ,
  \end{equation*}
  from which we can derive the first order derivatives, yielding
  \begin{equation}
    \label{eq:dyiB}
    \begin{aligned}
      \partial_{y_i} \Bbfm(\xbfm, \ybfm)
      &= \sigma_i \psibfm_i(\xbfm)
      \groupp[\bigg]{\psibfm_0(\xbfm)  + \sum_{k=1}^{M} \sigma_k \psibfm_k(\xbfm) y_k}^\trans \\
      &\qquad {} + \groupp[\bigg]{\psibfm_0(\xbfm)  + \sum_{k=1}^{M} \sigma_k \psibfm_k(\xbfm) y_k}
      \sigma_i \psibfm_i^\trans(\xbfm),
    \end{aligned}
  \end{equation}
  and from those also the second order derivatives. They are given by
  \begin{equation}
    \label{eq:dyjdyiB}
    \partial_{y_j} \partial_{y_i} \Bbfm(\xbfm, \ybfm)
    = \sigma_i \psibfm_i(\xbfm) \sigma_j \psibfm_j^\trans(\xbfm)
    + \sigma_j \psibfm_j(\xbfm) \sigma_i \psibfm_i^\trans(\xbfm) .
  \end{equation}
  Since the second order derivatives with respect to $\ybfm$ are constant,
  all higher order derivatives with respect to $\ybfm$ vanish.
  
  We obviously have $\normt{\Bbfm}_{d \times d} = \normt{\Vbfm}_{d}^2 \leq a_{\max}^2$.
  From \eqref{eq:dyiB} we can now derive the bound
  \begin{equation*}
    \normt[\big]{\partial_{y_i} \Bbfm}_{d \times d}
    \leq 2 \normt[\big]{\sigma_i \psibfm_i}_{d} \normt[\bigg]{\psibfm_0  + \sum_{k=1}^{M} \sigma_k \psibfm_k y_k}_{d}
    \leq 2 \gamma_i a_{\max}
  \end{equation*}
  and \eqref{eq:dyjdyiB} leads us to
  $\normt[\big]{\partial_{y_j} \partial_{y_i} \Bbfm}_{d \times d}
  \leq 2 \normt[\big]{\sigma_i \psibfm_i}_{d} \normt[\big]{\sigma_j \psibfm_j}_{d}\leq 2 \gamma_i \gamma_j$.
  Therefore, we have
  \begin{equation*}
    \normt[\big]{\partial_\ybfm^\alphabfm \Bbfm}_{d \times d}
    \leq \begin{cases}
      a_{\max}^2 \gammabfm^\alphabfm , & \text{if $\norms{\alphabfm} = 0$,} \\
      2 a_{\max} \gammabfm^\alphabfm , & \text{if $\norms{\alphabfm} = 1$,} \\
      2 \gammabfm^\alphabfm , & \text{if $\norms{\alphabfm} = 2$,} \\
      0 , & \text{if $\norms{\alphabfm} > 2$,}
    \end{cases}
  \end{equation*}
  and are finished since $a_{\max} \geq 1$.
\end{proof}
\begin{lemma}\label{lemma:DEybounds}
  Let us define
    \(C(\xbfm, \ybfm) \isdef \Vbfm^\trans(\xbfm, \ybfm) \Vbfm(\xbfm, \ybfm)\), 
    \(D(\xbfm, \ybfm) \isdef \groupp[\big]{C(\xbfm, \ybfm)}^{-1}\) and
    \(E(\xbfm, \ybfm) \isdef \sqrt{C(\xbfm, \ybfm)}\).
  Then, we know for all $\alphabfm \in \Nbbb_0^M$ that
  \begin{equation*}
    \normt{\partial_\ybfm^\alphabfm D}
    \leq \norms{\alphabfm}!\frac{1}{a_{\min}^2}
    \groupp[\bigg]{\frac{2 a_{\max}^2}{a_{\min}^2 \log 2}}^{\norms{\alphabfm}} \gammabfm^{\alphabfm}
  \quad\text{and}\quad
    \normt{\partial_\ybfm^\alphabfm E}
    \leq \norms{\alphabfm}!a_{\max}
    \groupp[\bigg]{\frac{2 a_{\max}^2}{a_{\min}^2 \log 2}}^{\norms{\alphabfm}} \gammabfm^{\alphabfm} .
  \end{equation*}
\end{lemma}
\begin{proof}
  The function $C$ can be expressed as
  \begin{equation*}
    C(\xbfm, \ybfm)
    = \groupp[\bigg]{\psibfm_0(\xbfm)  + \sum_{k=1}^{M} \sigma_k \psibfm_k(\xbfm) y_k}^\trans
    \groupp[\bigg]{\psibfm_0(\xbfm)  + \sum_{k=1}^{M} \sigma_k \psibfm_k(\xbfm) y_k} ,
  \end{equation*}
  which, by derivation, gives the following expressions for the first order derivatives,
  \begin{equation}
    \label{eq:dyiC}
    \begin{aligned}
      \partial_{y_i} C(\xbfm, \ybfm)
      &= \sigma_i \psibfm_i^\trans(\xbfm)
      \groupp[\bigg]{\psibfm_0(\xbfm)  + \sum_{k=1}^{M} \sigma_k \psibfm_k(\xbfm) y_k} \\
      &\quad {}+ \groupp[\bigg]{\psibfm_0(\xbfm)  + \sum_{k=1}^{M} \sigma_k \psibfm_k(\xbfm) y_k}^\trans
      \sigma_i \psibfm_i(\xbfm) .
    \end{aligned}
  \end{equation}
  Computing the second order derivatives then yields
  \begin{equation}
    \label{eq:dyjdyiC}
    \partial_{y_j} \partial_{y_i} C(\xbfm, \ybfm)
    = \sigma_i \psibfm_i^\trans(\xbfm) \sigma_j \psibfm_j(\xbfm)
    + \sigma_j \psibfm_j^\trans(\xbfm) \sigma_i \psibfm_i(\xbfm)
  \end{equation}
  and all higher order derivatives with respect to $\ybfm$ are zero,
  since the second order derivatives with respect to $\ybfm$ are already constant.
  
  We use \eqref{eq:Vellipticity} to arrive at
\(
    a_{\min}^2 \leq \normt{C} = \normt{\Vbfm}_{d}^2 \leq a_{\max}^2 ,
\)
  which also yields
  \begin{equation*}
    \frac{1}{a_{\max}^2} \leq \normt{D} \leq \frac{1}{a_{\min}^2}
    \quad\text{and}\quad
    a_{\min} \leq \normt{E} \leq a_{\max} .
  \end{equation*}
  Using \eqref{eq:dyiC} yields the bound
  \begin{equation*}
    \normt[\big]{\partial_{y_i} C}
    \leq 2 \normt[\big]{\sigma_i \psibfm_i}_{d} \normt[\bigg]{\psibfm_0  + \sum_{k=1}^{M} \sigma_k \psibfm_k y_k}_{d}
    \leq 2 \gamma_i a_{\max}
  \end{equation*}
  and, from \eqref{eq:dyjdyiC}, we can derive the bound
  $\normt[\big]{\partial_{y_j} \partial_{y_i} C}
  \leq 2 \normt[\big]{\sigma_i \psibfm_i}_{d} \normt[\big]{\sigma_j \psibfm_j}_{d}\leq 2 \gamma_i \gamma_j$.
  Thus, we know that
  \begin{equation*}
    \normt[\big]{\partial_\ybfm^\alphabfm C}
    \leq 2 a_{\max}^2 \gammabfm^\alphabfm .
  \end{equation*}
  
  Because $D = v \circ C$ with $v(x) = x^{-1}$ and $E = w \circ C$ with $w(x) = \sqrt{x}$ are
  composite functions,
  we employ the Fa\`a di Bruno formula, see \cite{CS}, to compute their derivatives.
  The $r$-th derivative of $v$ is given by
  \begin{equation*}
    \dif{x}{r}v(x) = (-1)^r r! x^{-1-r} = (-1)^r r! v(x)^{r+1}
  \end{equation*}
  and the $r$-th derivative of $w$ is given by
  \begin{equation*}
    \dif{x}{r}w(x) = c_r x^{\frac{1}{2}-r} = c_r w(x) v(x)^{r} ,
  \end{equation*}
  where $c_r \isdef \prod_{i=0}^{r-1} \groupp[\big]{\frac{1}{2} - i}$.
  For $n = \norms{\alphabfm}$ we thus arrive at
  \begin{equation}
    \label{eq:partialD}
    \partial_\ybfm^\alphabfm D(\xbfm, \ybfm)
    = \sum_{r=1}^{n} (-1)^r r! D(\xbfm, \ybfm)^{r+1}
    \sum_{P(\alphabfm, r)} \alphabfm! \prod_{j=1}^{n}
    \frac{\groupp[\Big]{\partial_\ybfm^{\betabfm_j} C(\xbfm, \ybfm)}^{k_j}}{k_j! (\betabfm_j!)^{k_j}}
  \end{equation}
  and
  \begin{equation}
    \label{eq:partialE}
    \partial_\ybfm^\alphabfm E(\xbfm, \ybfm)
    = \sum_{r=1}^{n} c_r E(\xbfm, \ybfm) D(\xbfm, \ybfm)^r
    \sum_{P(\alphabfm, r)} \alphabfm! \prod_{j=1}^{n}
    \frac{\groupp[\Big]{\partial_\ybfm^{\betabfm_j} C(\xbfm, \ybfm)}^{k_j}}{k_j! (\betabfm_j!)^{k_j}} ,
  \end{equation}
  where $P(\alphabfm, r)$ is a subset of integer partitions of a multiindex $\alphabfm$
  into $r$ non\hyp{}vanishing multiindices, given by
  \begin{align*}
    P(\alphabfm, r) \isdef
    \groupb[\bigg]{
    &\groupp[\Big]{(k_1, \betabfm_1), \ldots, (k_n, \betabfm_n)}
    \in \groupp[\Big]{\Nbbb_0 \times \Nbbb_0^M}^n :
    \sum_{j=1}^n k_j \betabfm_j = \alphabfm, \sum_{j=1}^n k_i = r, \\
    &\text{and there exists } 1 \leq s \leq n :
    k_j = 0 \text{ and } \betabfm_j = \zerobfm \text{ for all } 1 \leq j \leq n-s, \\
    & k_j > 0 \text{ for all } n-s+1 \leq j \leq n
    \text{ and } \zerobfm \prec \betabfm_{n-s+1} \prec \cdots \betabfm_n
    } .
  \end{align*}
  The relation $\betabfm \prec \betabfm'$ for multiindices $\betabfm, \betabfm' \in \Nbbb_0^M$ means
  that either $\norms{\betabfm} < \norms{\betabfm'}$ or,
  when $\norms{\betabfm} =\norms{\betabfm'}$,
  there exists $0 \leq k < m$ such that $\betabfm_1 = \betabfm_1', \ldots, \betabfm_k = \betabfm_k'$
  and $\betabfm_{k+1} < \betabfm_{k+1}'$.
  
  Taking the norm of \eqref{eq:partialD} and \eqref{eq:partialE} leads us to
  \begin{align*}
    \normt[\big]{\partial_\ybfm^\alphabfm D}
    &\leq \sum_{r=1}^{n} r! \normt{D}^{r+1}
    \sum_{P(\alphabfm, r)} \alphabfm! \prod_{j=1}^{n}
    \frac{\normt[\big]{\partial_\ybfm^{\betabfm_j} C}^{k_j}}{k_j! (\betabfm_j!)^{k_j}} \\
    &\leq \sum_{r=1}^{n} r! \groupp[\bigg]{\frac{1}{a_{\min}^2}}^{r+1}
    \sum_{P(\alphabfm, r)} \alphabfm! \prod_{j=1}^{n}
    \frac{\groupp[\big]{2 a_{\max}^2 \gammabfm^{\betabfm_j}}^{k_j}}{k_j! (\betabfm_j!)^{k_j}} \\
    &= \gammabfm^{\alphabfm} \sum_{r=1}^{n} r! \groupp[\bigg]{\frac{1}{a_{\min}^2}}^{r+1}
    \groupp[\big]{2 a_{\max}^2}^r \sum_{P(\alphabfm, r)} \alphabfm! \prod_{j=1}^{n}
    \frac{1}{k_j! (\betabfm_j!)^{k_j}}
  \end{align*}
  and
  \begin{align*}
    \normt[\big]{\partial_\ybfm^\alphabfm E}
    &\leq \sum_{r=1}^{n} \norms{c_r} \normt{E} \normt{D}^r
    \sum_{P(\alphabfm, r)} \alphabfm! \prod_{j=1}^{n}
    \frac{\normt[\big]{\partial_\ybfm^{\betabfm_j} C}^{k_j}}{k_j! (\betabfm_j!)^{k_j}} \\
    &\leq \sum_{r=1}^{n} \norms{c_r} a_{\max} \groupp[\bigg]{\frac{1}{a_{\min}^2}}^r
    \sum_{P(\alphabfm, r)} \alphabfm! \prod_{j=1}^{n}
    \frac{\groupp[\big]{2 a_{\max}^2 \gammabfm^{\betabfm_j}}^{k_j}}{k_j! (\betabfm_j!)^{k_j}} \\
    &= \gammabfm^{\alphabfm} \sum_{r=1}^{n} \norms{c_r} a_{\max} \groupp[\bigg]{\frac{1}{a_{\min}^2}}^r
    \groupp[\big]{2 a_{\max}^2}^r \sum_{P(\alphabfm, r)} \alphabfm! \prod_{j=1}^{n}
    \frac{1}{k_j! (\betabfm_j!)^{k_j}} .
  \end{align*}
  Since we know from \cite{CS} that
  \begin{equation*}
    \sum_{P(\alphabfm, r)} \alphabfm! \prod_{j=1}^{n} \frac{1}{k_j! (\betabfm_j!)^{k_j}} = S_{n,r} ,
  \end{equation*}
  where $S_{n,r}$ denotes the Stirling numbers of the second kind, see \cite{abra}, and
  that $\norms{c_r} \leq r!$, we can obtain
  \begin{equation*}
    \normt[\big]{\partial_\ybfm^\alphabfm D}
    \leq \frac{1}{a_{\min}^2} \gammabfm^{\alphabfm}
    \sum_{r=1}^{n} r! \groupp[\bigg]{\frac{2 a_{\max}^2}{a_{\min}^2}}^r S_{n,r}
    \leq \frac{1}{a_{\min}^2}
    \groupp[\bigg]{\frac{2 a_{\max}^2}{a_{\min}^2}}^{\norms{\alphabfm}} \gammabfm^{\alphabfm}
    \sum_{r=1}^{n} r! S_{n,r}
  \end{equation*}
  and
  \begin{equation*}
    \normt[\big]{\partial_\ybfm^\alphabfm E}
    \leq a_{\max} \gammabfm^{\alphabfm}
    \sum_{r=1}^{n} r! \groupp[\bigg]{\frac{2 a_{\max}^2}{a_{\min}^2}}^r S_{n,r}
    \leq a_{\max}
    \groupp[\bigg]{\frac{2 a_{\max}^2}{a_{\min}^2}}^{\norms{\alphabfm}} \gammabfm^{\alphabfm}
    \sum_{r=1}^{n} r! S_{n,r} .
  \end{equation*}
  Because $\sum_{r=1}^{n} r! S_{n,r}$ equals the $n$-th ordered Bell number,
  we can bound it, see \cite{Beck}, by
  \begin{equation*}
    \sum_{r=1}^{n} r! S_{n,r} \leq \frac{n!}{(\log 2)^n}.
  \end{equation*}
  This implies the assertion.
\end{proof}
By combining the previous two lemmata, we derive the following result.
\begin{lemma}\label{lemma:Fybounds}
  We define $\Fbfm$ by
  \begin{equation*}
    \Fbfm(\xbfm, \ybfm)
    \isdef \frac{\Vbfm(\xbfm, \ybfm) \Vbfm^\trans(\xbfm, \ybfm)}{\Vbfm^\trans(\xbfm, \ybfm) \Vbfm(\xbfm, \ybfm)} .
  \end{equation*}
  Then, we have for all $\alphabfm \in \Nbbb_0^M$ that
  \begin{equation*}
    \normt[\big]{\partial_\ybfm^\alphabfm \Fbfm}_{d \times d}
    \leq \norms{\alphabfm}! \frac{6 a_{\max}^2}{a_{\min}^2}
    \groupp[\bigg]{\frac{2 a_{\max}^2}{a_{\min}^2 \log 2}}^{\norms{\alphabfm}} \gammabfm^{\alphabfm} .
  \end{equation*}
\end{lemma}
\begin{proof}
  We can equivalently state $\Fbfm$ as $\Fbfm(\xbfm, \ybfm) = D(\xbfm, \ybfm) \Bbfm(\xbfm, \ybfm)$.
  Then, by applying the Leibniz rule, we arrive at
  \begin{equation*}
    \partial_\ybfm^\alphabfm \Fbfm(\xbfm, \ybfm)
    = \sum_{\betabfm \leq \alphabfm} \binom{\alphabfm}{\betabfm} 
    \groupp[\Big]{\partial_\ybfm^\betabfm D(\xbfm, \ybfm)}
    \groupp[\Big]{\partial_\ybfm^{\alphabfm-\betabfm} \Bbfm(\xbfm, \ybfm)} .
  \end{equation*}
  Taking the norm and using the bounds from Lemma~\ref{lemma:Bybounds}
  and Lemma~\ref{lemma:DEybounds} leads us to
  \begin{align*}
    \normt[\big]{\partial_\ybfm^\alphabfm \Fbfm}_{d \times d}
    &\leq \sum_{\betabfm \leq \alphabfm} \binom{\alphabfm}{\betabfm} 
    \normt[\big]{\partial_\ybfm^\betabfm D}
    \normt[\big]{\partial_\ybfm^{\alphabfm-\betabfm} \Bbfm}_{d \times d} \\
    &\leq \sum_{\betabfm \leq \alphabfm} \binom{\alphabfm}{\betabfm} 
    \norms{\betabfm}! \frac{1}{a_{\min}^2}
    \groupp[\bigg]{\frac{2 a_{\max}^2}{a_{\min}^2 \log 2}}^{\norms{\betabfm}} \gammabfm^{\betabfm}
    2 a_{\max}^2 \gammabfm^{\alphabfm-\betabfm} \\
    &\leq \frac{2 a_{\max}^2}{a_{\min}^2}
    \groupp[\bigg]{\frac{2 a_{\max}^2}{a_{\min}^2 \log 2}}^{\norms{\alphabfm}} \gammabfm^{\alphabfm}
    \sum_{\betabfm \leq \alphabfm} \binom{\alphabfm}{\betabfm} \norms{\betabfm}! .
  \end{align*}
  Lastly, the combinatorial identity
  \begin{equation}
    \label{eq:combinatorialidentity}
    \sum_{\substack{\betabfm \leq \alphabfm\\\norms{\betabfm}=j}} \binom{\alphabfm}{\betabfm}
    = \binom{\norms{\alphabfm}}{j}
  \end{equation}
  yields the bound
  \begin{equation*}
    \sum_{\betabfm \leq \alphabfm} \binom{\alphabfm}{\betabfm} \norms{\betabfm}!
    = \sum_{j=0}^{\norms{\alphabfm}} j!
    \sum_{\substack{\betabfm \leq \alphabfm\\\norms{\betabfm}=j}} \binom{\alphabfm}{\betabfm}
    = \sum_{j=0}^{\norms{\alphabfm}} j! \binom{\norms{\alphabfm}}{j}
    = \norms{\alphabfm}! \sum_{k=0}^{\norms{\alphabfm}} \frac{1}{k!}
    \leq 3 \norms{\alphabfm}!. 
  \end{equation*}
\end{proof}
Gathering all the regularity estimates for the different terms in \eqref{eq:AdmV} gives us 
the regularity of the diffusion matrix \({\bf A}\).

\begin{theorem}\label{theorem:Aybounds}
  The derivatives of the diffusion matrix $\Abfm$ defined in \eqref{eq:AdmV} satisfy
  \begin{equation*}
    \normt[\big]{\partial_\ybfm^\alphabfm \Abfm}_{d \times d}
    \leq (\norms{\alphabfm} + 1)! (2 a_{\max})
    \frac{6 a_{\max}^2}{a_{\min}^2}
    \groupp[\bigg]{\frac{2 a_{\max}^2}{a_{\min}^2 \log 2}}^{\norms{\alphabfm}} \gammabfm^{\alphabfm}
  \end{equation*}
  for all $\alphabfm \in \Nbbb_0^M$ with $\norms{\alphabfm} \geq 1$.
\end{theorem}
\begin{proof}
  We can state $\Abfm$ as
  $\Abfm(\xbfm, \ybfm) = a \Ibfm + E(\xbfm, \ybfm) \Fbfm(\xbfm, \ybfm) - a \Fbfm(\xbfm, \ybfm)$,
  which, with the Leibniz rule, yields
  \begin{equation*}
    \partial_\ybfm^\alphabfm \Abfm(\xbfm, \ybfm)
    = \sum_{\betabfm \leq \alphabfm} \binom{\alphabfm}{\betabfm} 
    \groupp[\Big]{\partial_\ybfm^\betabfm E(\xbfm, \ybfm)}
    \groupp[\Big]{\partial_\ybfm^{\alphabfm-\betabfm} \Fbfm(\xbfm, \ybfm)}
    - a \partial_\ybfm^\alphabfm \Fbfm(\xbfm, \ybfm) .
  \end{equation*}
  Then, by taking the norm and
  inserting the bounds from Lemmata~\ref{lemma:DEybounds} and \ref{lemma:Fybounds},
  we arrive at
  \begin{align*}
    \normt[\big]{\partial_\ybfm^\alphabfm \Abfm}_{d \times d}
    &\leq \sum_{\betabfm \leq \alphabfm} \binom{\alphabfm}{\betabfm} 
    \normt[\big]{\partial_\ybfm^\betabfm E}
    \normt[\big]{\partial_\ybfm^{\alphabfm-\betabfm} \Fbfm}_{d \times d}
    + a_{\max} \normt[\big]{\partial_\ybfm^\alphabfm \Fbfm}_{d \times d} \\ 
    &\leq \sum_{\betabfm \leq \alphabfm} \binom{\alphabfm}{\betabfm} 
    \norms{\betabfm}! a_{\max}
    \groupp[\bigg]{\frac{2 a_{\max}^2}{a_{\min}^2 \log 2}}^{\norms{\betabfm}} \gammabfm^{\betabfm} \\
    &\qquad\qquad \norms{\alphabfm-\betabfm}! \frac{6 a_{\max}^2}{a_{\min}^2}
    \groupp[\bigg]{\frac{2 a_{\max}^2}{a_{\min}^2 \log 2}}^{\norms{\alphabfm-\betabfm}}
    \gammabfm^{\alphabfm-\betabfm} \\
    &\quad {} + \norms{\alphabfm}! a_{\max} \frac{6 a_{\max}^2}{a_{\min}^2}
    \groupp[\bigg]{\frac{2 a_{\max}^2}{a_{\min}^2 \log 2}}^{\norms{\alphabfm}} \gammabfm^{\alphabfm} \\ 
    &\leq a_{\max} \frac{6 a_{\max}^2}{a_{\min}^2}
    \groupp[\bigg]{\frac{2 a_{\max}^2}{a_{\min}^2 \log 2}}^{\norms{\alphabfm}} \gammabfm^{\alphabfm}
    \sum_{\betabfm \leq \alphabfm} \binom{\alphabfm}{\betabfm}
    \norms{\betabfm}! \norms{\alphabfm-\betabfm}! \\
    &\quad {} + a_{\max} \frac{6 a_{\max}^2}{a_{\min}^2}
    \groupp[\bigg]{\frac{2 a_{\max}^2}{a_{\min}^2 \log 2}}^{\norms{\alphabfm}} \gammabfm^{\alphabfm}
    \norms{\alphabfm}! .
  \end{align*}
  Finally, the combinatorial identity \eqref{eq:combinatorialidentity} yields, see eg.\ \cite{HPS14},
  \begin{equation*}
    \sum_{\betabfm \leq \alphabfm} \binom{\alphabfm}{\betabfm}
    \norms{\betabfm}! \norms{\alphabfm-\betabfm}!
    = (\norms{\alphabfm} + 1)! . 
  \end{equation*}
\end{proof}

If we now define the modified sequence $\mubfm = (\mu_k)_{k \in \Nbbb_0}$ as
\begin{equation*}
  \mu_k \isdef \frac{4 a_{\max}^2}{a_{\min}^2 \log 2} \gamma_k
\quad\text{and also}\quad
  c_\Abfm \isdef (2 a_{\max}) \frac{6 a_{\max}^2}{a_{\min}^2},
\end{equation*}
we can summarize the results attained so far by\footnote{Note 
that the additional factor of $2$ in $\mu_k$ removes the factor 
$\norms{\alphabfm}+1$ from the factorial expression, since we 
know that $2^{\norms{\alphabfm}} \geq \norms{\alphabfm}+1$.}
\begin{equation*}
  \normt[\big]{\partial_\ybfm^\alphabfm \Abfm}_{d \times d}
  \leq \norms{\alphabfm}! c_\Abfm \mubfm^{\alphabfm}.
\end{equation*}

\subsection{Parametric regularity of the solution}
The above regularity estimate carries over to the solution 
with slightly different constants.
\begin{theorem}\label{theorem:uybounds}
  For almost every $\ybfm \in \square$,
  the derivatives of the solution $u(\ybfm)$ of \eqref{eq:smwpoissonsw} satisfy
  \begin{equation*}
    \norm[\big]{\partial_\ybfm^\alphabfm u(\ybfm)}_{H^1(D)} 
    \leq \norms{\alphabfm}! \mubfm^\alphabfm
    \groupp[\bigg]{\frac{a_{\max}}{a_{\min} c_V^2} \max\groupb[\Big]{2 c_\Abfm, \norm{f}_{\widetilde{H}^{-1}(D)} + \norm{g}_{H^{-1/2}(\Gamma_N)}}}^{\norms{\alphabfm}+1} .
  \end{equation*}
\end{theorem}
\begin{proof}
  By differentiation of the variational formulation \eqref{eq:smwpoissonsw} with respect to $\ybfm$
  we arrive, for arbitrary $v \in V$, at
  \begin{equation*}
    \groupp[\Big]{\partial_\ybfm^\alphabfm \groupp[\big]{\Abfm(\ybfm) \Grad_\xbfm u(\ybfm)},
    \Grad_\xbfm v}_{L^2(D; \Rbbb^d)} = 0 .
  \end{equation*}
  Applying the Leibniz rule on the left\hyp{}hand side yields
  \begin{equation*}
    \groupp[\bigg]{\sum_{\betabfm \leq \alphabfm} \binom{\alphabfm}{\betabfm}
    \partial_\ybfm^{\alphabfm-\betabfm} \Abfm(\ybfm)
    \partial_\ybfm^{\betabfm} \Grad_\xbfm u(\ybfm) ,
    \Grad_\xbfm v}_{L^2(D; \Rbbb^d)} = 0 .
  \end{equation*}
  Then, by rearranging and using the linearity of the gradient, we find
  \begin{equation*}
    \int_D \groupa[\Big]{\Abfm(\ybfm) \Grad_\xbfm \partial_\ybfm^\alphabfm u(\ybfm),
    \Grad_\xbfm v} \diff{\xbfm}
    = - \sum_{\betabfm < \alphabfm} \binom{\alphabfm}{\betabfm}
    \int_D \groupa[\Big]{\partial_\ybfm^{\alphabfm-\betabfm} \Abfm(\ybfm)
    \Grad_\xbfm \partial_\ybfm^{\betabfm} u(\ybfm) ,
    \Grad_\xbfm v} \diff{\xbfm} .
  \end{equation*}
  We now choose $v = \partial_\ybfm^\alphabfm u(\ybfm)$ and
  employ the coercivity as well as the bound from Theorem~\ref{theorem:Aybounds}.
  This results in
  \begin{align*}
    a_{\min} c_V^2 \norm[\big]{\partial_\ybfm^\alphabfm u(\ybfm)}_{H^1(D)}^2
    &\leq - \sum_{\betabfm < \alphabfm} \binom{\alphabfm}{\betabfm}
    \int_D \groupa[\Big]{\partial_\ybfm^{\alphabfm-\betabfm} \Abfm(\ybfm)
    \Grad_\xbfm\partial_\ybfm^{\betabfm} u(\ybfm) ,
    \Grad_\xbfm \partial_\ybfm^\alphabfm u(\ybfm)} \diff{\xbfm} \\
    &\leq \sum_{\betabfm < \alphabfm} \binom{\alphabfm}{\betabfm}
    \normt[\big]{\partial_\ybfm^{\alphabfm-\betabfm} \Abfm}_{d \times d}
    \norm[\big]{\partial_\ybfm^{\betabfm} u(\ybfm)}_{H^1(D)}
    \norm[\big]{\partial_\ybfm^\alphabfm u(\ybfm)}_{H^1(D)} \\
    &\leq \sum_{\betabfm < \alphabfm} \binom{\alphabfm}{\betabfm}
    \norms{\alphabfm-\betabfm}! c_\Abfm \mubfm^{\alphabfm-\betabfm}
    \norm[\big]{\partial_\ybfm^{\betabfm} u(\ybfm)}_{H^1(D)}
    \norm[\big]{\partial_\ybfm^\alphabfm u(\ybfm)}_{H^1(D)} ,
  \end{align*}
  from which we derive
  \begin{equation*}
    \norm[\big]{\partial_\ybfm^\alphabfm u(\ybfm)}_{H^1(D)}
    \leq \frac{c}{2}
    \sum_{\betabfm < \alphabfm} \binom{\alphabfm}{\betabfm}
    \norms{\alphabfm-\betabfm}! \mubfm^{\alphabfm-\betabfm}
    \norm[\big]{\partial_\ybfm^{\betabfm} u(\ybfm)}_{H^1(D)} ,
  \end{equation*}
  where
  \begin{equation*}
    c \isdef \frac{a_{\max}}{a_{\min} c_V^2} \max\groupb[\Big]{2 c_\Abfm, \norm{f}_{\widetilde{H}^{-1}(D)} + \norm{g}_{H^{-1/2}(\Gamma_N)}} .
  \end{equation*}
  We note that, by definition of $c$, we have $c \geq 2$
  and furthermore, because of Lemma~\ref{lemma:ubound}, we also have that
  $\norm[\big]{u(\ybfm)}_{H^1(D)} \leq c$,
  which means that the assertion is true for $\norms{\alphabfm} = 0$.

  Thus, we can use an induction over $\norms{\alphabfm}$ to prove the hypothesis
  \begin{equation*}
    \norm[\big]{\partial_\ybfm^\alphabfm u(\ybfm)}_{H^1(D)} 
    \leq \norms{\alphabfm}! \mubfm^\alphabfm c^{\norms{\alphabfm}+1}
  \end{equation*}
  for $\norms{\alphabfm} > 0$.
  Let the assertions hold for all $\alphabfm$,
  which satisfy $\norms{\alphabfm} \leq n-1$ for some $n \geq 1$.
  Then, we know for all $\alphabfm$ with $\norms{\alphabfm} = n$ that
  \begin{align*}
    \norm[\big]{\partial_\ybfm^\alphabfm u(\ybfm)}_{H^1(D)}
    &\leq \frac{c}{2}
    \sum_{\betabfm < \alphabfm} \binom{\alphabfm}{\betabfm}
    \norms{\alphabfm-\betabfm}! \mubfm^{\alphabfm-\betabfm}
    \norm[\big]{\partial_\ybfm^{\betabfm} u(\ybfm)}_{H^1(D)} \\
    &\leq \frac{c}{2} \mubfm^\alphabfm
    \sum_{\betabfm < \alphabfm} \binom{\alphabfm}{\betabfm}
    \norms{\alphabfm-\betabfm}! \norms{\betabfm}!
    c^{\norms{\betabfm}+1} \\
    &= \frac{c}{2} \mubfm^\alphabfm
    \sum_{j=0}^{n-1} \sum_{\substack{\betabfm < \alphabfm\\\norms{\betabfm}=j}} \binom{\alphabfm}{\betabfm}
    \norms{\alphabfm-\betabfm}! \norms{\betabfm}!
    c^{\norms{\betabfm}+1} .
  \end{align*}
  Making use of the combinatorial identity \eqref{eq:combinatorialidentity} yields
  \begin{align*}
    \norm[\big]{\partial_\ybfm^\alphabfm u(\ybfm)}_{H^1(D)}
    &\leq \frac{c}{2} \mubfm^\alphabfm
    \sum_{j=0}^{n-1} \binom{\norms{\alphabfm}}{j}
    (\norms{\alphabfm}-j)! j! c^{j+1} \\
    &= \frac{c}{2} \norms{\alphabfm}! \mubfm^\alphabfm
    c \sum_{j=0}^{n-1} c^j 
    \leq \frac{c}{2} \norms{\alphabfm}! \mubfm^\alphabfm
    c \frac{c^{\norms{\alphabfm}}}{c-1} 
    \leq \frac{c}{2(c-1)} \norms{\alphabfm}! \mubfm^\alphabfm
    c^{\norms{\alphabfm}+1} .
  \end{align*}
  Now, since $c \geq 2$, we have $c \leq 2(c-1)$ and hence also
  \begin{equation*}
    \norm[\big]{\partial_\ybfm^\alphabfm u(\ybfm)}_{H^1(D)}
    \leq \norms{\alphabfm}! \mubfm^\alphabfm c^{\norms{\alphabfm}+1} .
  \end{equation*}
  This completes the proof.
\end{proof}

\subsection{Numerical quadrature in the parameter}
Because of the regularity estimates shown before,
we can refer to \cite[Lemma 7]{HPS16}, which is a 
straightforward consequence from the results in \cite{Wan02},
for the convergence rate {of} the quasi\hyp{}Monte Carlo 
method (QMC) based on the Halton\hyp{}points. Therefore, under 
the assumptions made there, we can conclude that, for any 
$\delta > 0$, there is a constant $C_\delta$ such that
\begin{equation*}
  \norm[\bigg]{\Mean[u] -\frac{1}{N}\sum_{i=1}^N 
  	u({\boldsymbol\xi}_i)}_{H^1(D)} \leq C_\delta N^{\delta-1} .
\end{equation*}
A similar result {also} accounts for the variance \(\mathbb{V}[u]\), 
see eg.~\cite{HPS16x}.

For the sparse grid quadrature (SG), assume that 
\(\gamma_k\leq c k^{-r}\) for some constants \(c,r>0\).
Then, the anisotropic sparse Gauss-Legendre quadrature 
on level \(q\) with \(N(q)\) points satisfies the error estimate
\[
\|\Mean[u]-\mathcal{A}_{\bf w}(q,M)u\|_{H^1_0(D)}
\leq C N(q)^{-r/(2\log\log M)}\|u\|_{C(\square; H^1_0(D))}
\]
with a constant \(C>0\), see \cite{HHPS15}. Herein, we have 
\(w_k\isdef\log\big(\frac{1}{\gamma_k}+\sqrt{1+1/\gamma_k^2}\big)\), 
see eg.\ \cite{BNT}, and
\[
\mathcal{A}_{\bf w}(q,M)\isdef\sum_{{\boldsymbol\alpha}\in Y_{\bf w}(q,M)}
	c_{\bf w}({\boldsymbol\alpha}){\bf Q}_{\boldsymbol\alpha}\quad\text{with}
\quad c_{\bf w}({\boldsymbol\alpha})\isdef\sum_{{{\boldsymbol\beta}\in\{0,1\}^M} 
	\atop \langle{\boldsymbol\alpha}+{\boldsymbol\beta},{\bf w}\rangle\leq q} (-1)^{|\boldsymbol\beta|},
\]
where \({\bf Q}_{\boldsymbol\alpha}\) denotes the tensor 
product Gauss-Legendre quadrature operator of degree 
\(\lceil\boldsymbol\alpha/2\rceil\).\footnote{Note that
the quadrature operator \(\mathcal{A}_{\bf w}(q,M)\) rather 
refers to the sparse grid combination technique than the 
actual sparse grid quadrature operator.} The 
set \(Y_{\bf w}(q,M)\) is given according to
\[
 Y_{\bf w}(q,M)\isdef\big\{{\boldsymbol\alpha}\in\mathbb{N}_0^M:
 	q-\|{\bf w}\|_1\leq\langle{\boldsymbol\alpha},{\bf w}\rangle\leq q\big\}.
\]
With similiar arguments as in \cite{HPS16x}, the convergence 
result {again} carries over to \(\mathbb{V}[u]\).

\section{Numerical results}
We will now consider two examples of the model problem 
\eqref{eq:smwpoisson} with a diffusion coefficient of form 
\eqref{eq:AdmV} using the unit cube $D \isdef (0, 1)^3$ as 
the domain of computations. 
In both examples, we set 
the global strength \(a\) to $a \isdef 0.12$ and, for convenience, segment the boundary 
$\partial D$ into three disjoint parts:
\begin{align*}
  \Gamma_0 &\isdef \{1\} \times (0, 1) \times (0, 1) , \\
  \Gamma_1 &\isdef \{0\} \times (0, 1) \times (0, 1) \text{ and} \\
  \Gamma_2 &\isdef \partial D \setminus (\Gamma_0 \cup \Gamma_1) .
\end{align*}

Moreover, we choose 
the description of $\Vbfm$ to be defined by $\Mean[\Vbfm](\xbfm)
\isdef \begin{bmatrix} 1 & 0 & 0 \end{bmatrix}^\trans$ and
\begin{equation*}
      \Cov[\Vbfm](\xbfm, \xbfm')
      \isdef 0.01\exp\groupp[\bigg]{-\frac{\norm[\big]{\xbfm-\xbfm'}_2^2}{50}} 
      	\begin{bmatrix} 1& 0 & 0 \\ 0 & 9s_2(\xbfm, \xbfm')  & 0 \\ 0 & 0 & 9s_3(\xbfm, \xbfm')  \end{bmatrix},
\end{equation*}
where
\[
s_j(\xbfm, \xbfm')\isdef 16\cdot x_j(1-x_j)\cdot x_j'(1-x_j').
\]
The effect of the function \(s_j\) is to suppress the covariance along the normal direction 
on the boundary \(\Gamma_2\).
Some samples of the normalized vector field $\Vbfm/\|\Vbfm\|_2$ used for our examples,
which are computed on the level $3$ discretisation, are shown in 
Figure~\ref{fig:exp12V} as stream traces. By their definition, the 
stream traces are tracing our notable diffusion direction.

The numerical implementation is performed with aid of the 
problem\hyp{}solving environment DOLFIN \cite{FEniCSbook},
which is a part of the FEniCS Project \cite{FEniCSbook}. The 
Karhunen-Lo\`eve expansion of the vector field \({\bf V}\) is computed
by the pivoted Cholesky decomposition, see \cite{HPS,HPS14} for
the details. For the finite element discretization, we employ a sequence 
of triangulations $\Tcal_l$, subsequently we will call the index $l$ the 
level, yielded by successive uniform refinement, ie.\ cutting each 
tetrahedron into $8$ tetrahedra. Level 0 consists of \(6\cdot 2^3=48\) tetrahedra.
Then, we use element\hyp{}wise constant functions and the truncated 
pivoted Cholesky decomposition for the Karhunen\hyp{}Lo\`eve expansion 
approximation and continuous element\hyp{}wise linear functions in space.
The truncation criterion for the pivoted Cholesky decomposition is that the 
relative trace error is smaller than $10^{-4}\cdot 4^{-l}$; see Table~\ref{tab:nrealis3d} for the resulting parameter
dimensions \(M\). 
Since the exact solutions of the examples are unknown,
the errors will have to be estimated.
Therefore, in this section, we will estimate the errors for the levels $0$ to $5$
by substituting the exact solution with
the approximate solution computed on the level $6$ triangulation $\Tcal_6$
using {the quasi\hyp{}Monte Carlo quadrature based on Halton points with} $10^4$ samples.

For every level, we also define 
the number of samples used by the different quadrature methods.
For the quasi\hyp{}Monte Carlo method based on Halton points, we choose
\begin{equation*}
  N_l^{\operatorname{QMC}} \isdef \groupc[\Big]{2^{l/(1-\delta)} \cdot 10}
\end{equation*}
with $\delta \isdef 0.2$; see Table~\ref{tab:nrealis3d} for the resulting 
values of $N_l$. For the sparse grid quadrature, we set \(q_l=2l+2\).
Based on these choices, we expect to see an asymptotic rate of 
convergence of $2^{-l}$ in the $H^1$\hyp{}norm for the mean 
and in the $W^{1,1}$\hyp{}norm for the variance; see again Table~\ref{tab:nrealis3d} for the resulting 
values of $N_l^{\operatorname{SG}}\isdef N(q_l)$.
As a validation for the reference solution, we consider here also the convergence 
of a Monte Carlo quadrature{, using \(N_l^{\operatorname{MC}}\isdef 4^l\) samples on the level \(l\),} with respect to this reference. 
Note that to obtain an approximation for the mean square error, we average five realisations of the 
Monte Carlo (MC) estimator.
\begin{table}[htb]
\centering
\begin{tabular}{lrrrrrr}
  \toprule
  $l$ & $0$ & $1$ & $2$ & $3$ & $4$ & $5$ \\
  \midrule
  $N_l^{\operatorname{QMC}}$ & $10$ & $24$ & $57$ & $135$ & $320$ & $762$ \\
  $N_l^{\operatorname{SG}}$ & $1$ & $7$ & $29$ & $87$ & $265$ & $909$ \\
  $N_l^{\operatorname{MC}}$ & $1$ & $4$ & $16$ & $64$ & $256$ & $1024$ \\

  \midrule
    $M$ & $16$ & $24$ & $28$ & $34$& $44$& $53$ \\
  \bottomrule
\end{tabular}
\caption{The number of samples for the first six levels and the respective parameter dimensions.}
\label{tab:nrealis3d}
\end{table}

\subsection{Example}
In the first example, we set the source to \(f(\xbfm)\equiv 1\) and consider homogeneous Dirichlet data, ie.\
we set \(\Gamma_D=\Gamma_0\cup\Gamma_1\cup\Gamma_2\) and \(\Gamma_N=\emptyset\).

Visualisations of the reference solution's mean and variance are shown in 
Figure~\ref{fig:exp1mv}. Note that, to enable a view of the inside,
all data with coordinates $[x_1, x_2, x_3]^\trans$ such that $x_2+x_3 > 1$ are clipped.

\begin{figure}[h]
\centering
\begin{tikzpicture}[baseline]
  \begin{axis}[
      name=plot1,
      title={Mean},
      xlabel={$x_3$},
      ylabel={$x_1$},
      zlabel={$x_2$},
      zlabel style = {rotate=-90},
      width=6cm,
      height=6cm,
      scale mode=scale uniformly,
    ]
    \addplot3 graphics[
      points={
        (1,0,0) => (0,149)
        (0,1,0) => (1577,258)
        (0,0,1) => (577,1523)
        (1,1,1) => (1004,1119)
        (1,1,0)
        (1,0,1)
        (0,1,1)
        (0,0,0)
    }] {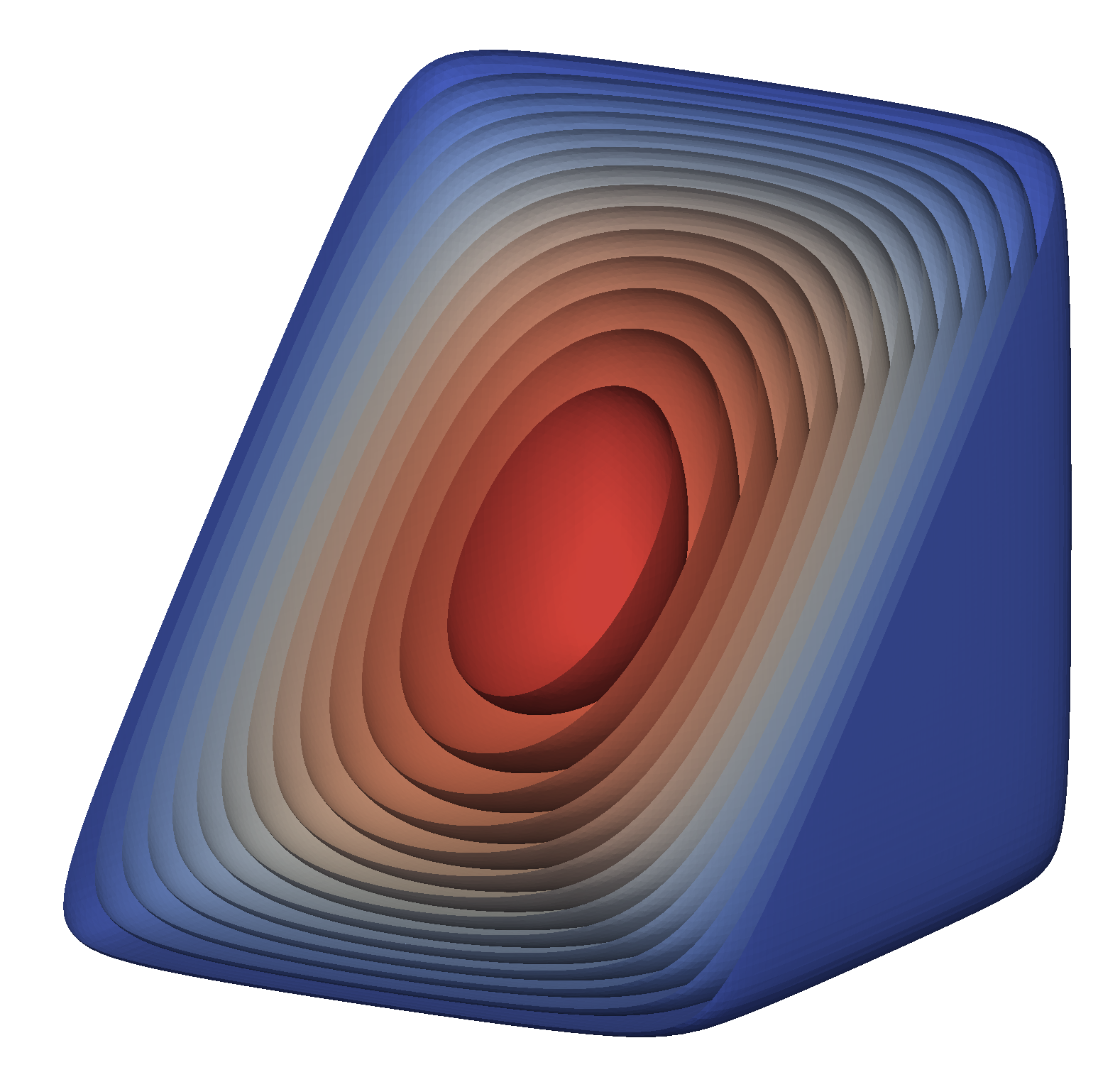};
  \end{axis}
  \begin{axis}[
      name=plot2,
      at=(plot1.below south east), anchor=above north east,
      yshift=-2ex,
      enlargelimits=false,
      axis on top,
      ytick=\empty,
      width=6cm,
      height=2cm,
    ] \addplot graphics[
      ymin=0, ymax=0.1,
      xmin=0.0, xmax=0.129993,
    ] {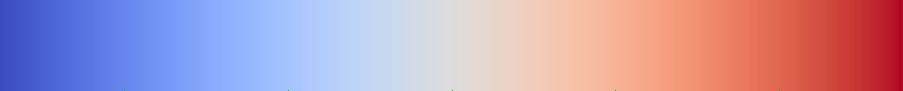};
  \end{axis}
\end{tikzpicture}%
\hspace*{0.6cm}%
\begin{tikzpicture}[baseline]
  \begin{axis}[
      name=plot1,
      title={Variance},
      xlabel={$x_3$},
      ylabel={$x_1$},
      zlabel={$x_2$},
      zlabel style = {rotate=-90},
      zticklabel pos=right,
      width=6cm,
      height=6cm,
      scale mode=scale uniformly,
    ]
    \addplot3 graphics[
      points={
        (1,0,0) => (0,149)
        (0,1,0) => (1577,258)
        (0,0,1) => (577,1523)
        (1,1,1) => (1004,1119)
        (1,1,0)
        (1,0,1)
        (0,1,1)
        (0,0,0)
    }] {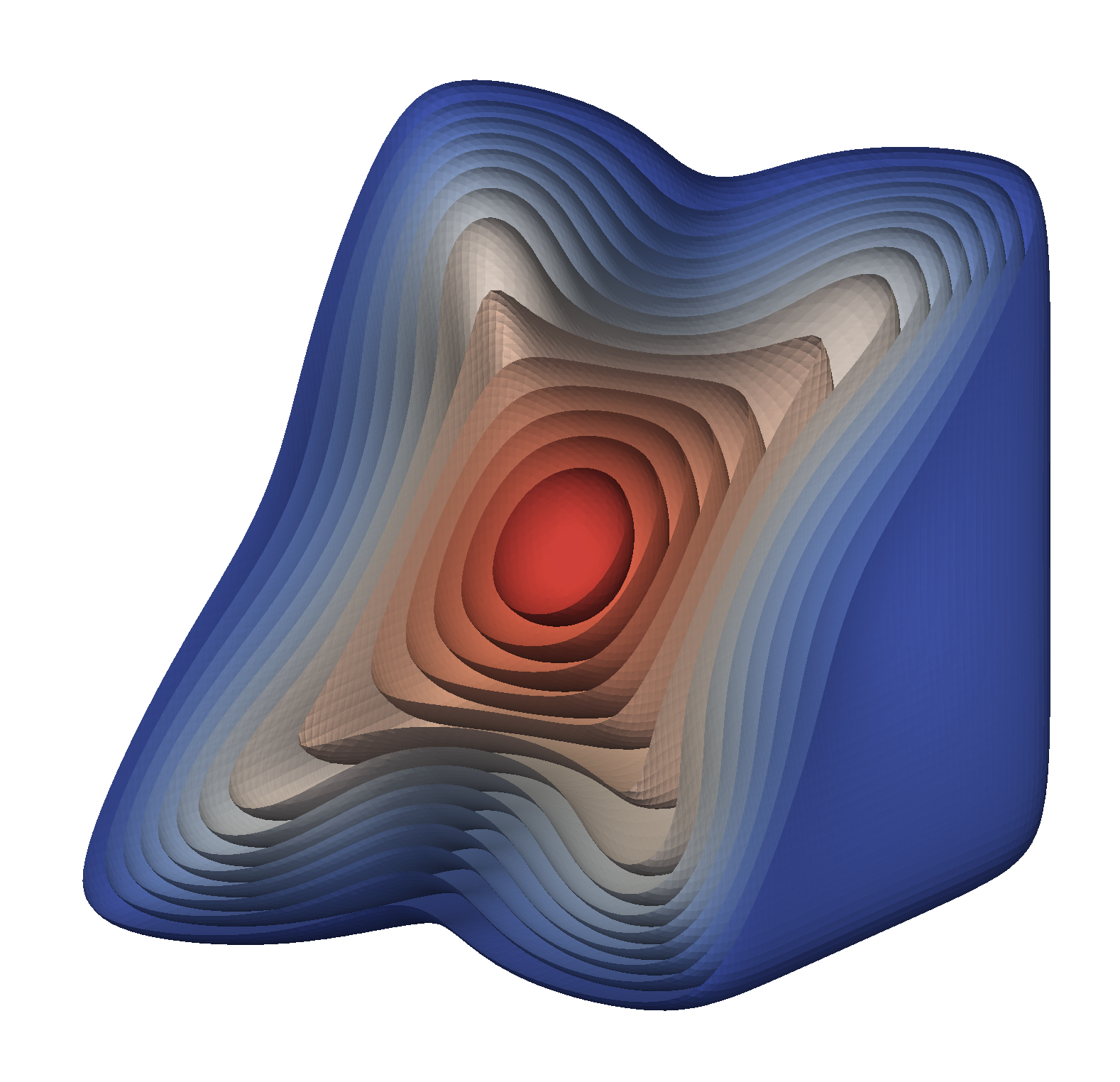};
  \end{axis}
  \begin{axis}[
      name=plot2,
      at=(plot1.below south east), anchor=above north east,
      yshift=-2ex,
      enlargelimits=false,
      axis on top,
      ytick=\empty,
      width=6cm,
      height=2cm,
    ] \addplot graphics[
      ymin=0, ymax=0.1,
      xmin=0.0, xmax=0.000192880,
    ] {./colorbar.png};
  \end{axis}
\end{tikzpicture}
\caption{Mean and variance of the solution.}
\label{fig:exp1mv}
\end{figure}

Figures~\ref{fig:exp1errormean} and~\ref{fig:exp1errorvar} show the estimated errors of the solution's 
mean on the left hand side and of the solution's variance on the right hand 
side, each versus the discretisation level for the different quadrature methods.
As expected, each of the quadrature methods achieves the predicted rate of convergence, however
QMC and SG provide slightly better errors in case of the variance.

\begin{figure}[htb]
\centering
\begin{minipage}{0.49\textwidth}
\centering
\begin{tikzpicture}[baseline]
  \begin{semilogyaxis}[
      xlabel={Level $l$},
      ylabel={Estimated error},
      xmajorgrids=true,
      yminorgrids=true,
      ymin=1e-2,
    ]
    \addplot[red, mark=o, mark size=1mm,line width=1pt] table [x=l,y=mH1] {./exp1errQMC.dat};
    \addlegendentry{QMC}
    \addplot[blue, mark=square, mark size=1mm,line width=1pt] table [x=l,y=mH1] {./exp1errSQ.dat};
    \addlegendentry{SG}
    \addplot[green!60!black, mark=diamond, mark size=1mm,line width=1pt] table [x=l,y=mH1] {./exp1errMC.dat};
    \addlegendentry{MC}
    \addplot[black, dashed] table [x=l,y=mH1a1] {./exp1errMC.dat};
    \addlegendentry{slope $2^{-l}$}
  \end{semilogyaxis}
\end{tikzpicture}
\caption{\(H^1\)-error in the mean.}
\label{fig:exp1errormean}
\end{minipage}
\begin{minipage}{0.49\textwidth}
\centering
\begin{tikzpicture}[baseline]
  \begin{semilogyaxis}[
      xlabel={Level $l$},
      ylabel={Estimated error},
      xmajorgrids=true,
      yminorgrids=true,
            yticklabel pos=right,
    ]
    \addplot[red, mark=o, mark size=1mm,line width=1pt] table [x=l,y=vW11] {./exp1errQMC.dat};
    \addlegendentry{QMC}
    \addplot[blue, mark=square, mark size=1mm,line width=1pt] table [x=l,y=vW11] {./exp1errSQ.dat};
    \addlegendentry{SG}
    \addplot[green!60!black, mark=diamond, mark size=1mm,line width=1pt] table [x=l,y=vW11] {./exp1errMC.dat};
    \addlegendentry{MC}
    \addplot[black, dashed] table [x=l,y=vW11a1] {./exp1errMC.dat};
    \addlegendentry{slope $2^{-l}$}
  \end{semilogyaxis}
\end{tikzpicture}
\caption{\(W^{1,1}\)-error in the variance.}
\label{fig:exp1errorvar}
\end{minipage}
\end{figure}

\subsection{Example}
The data in this example are given as follows. We remove the source, ie.\ \(f(\xbfm)\equiv 0\), and consider
\begin{equation*}
      \Gamma_N \isdef \Gamma_0 \cup \Gamma_1
      \quad \text{with} \quad
      g(\xbfm) \isdef \begin{cases} 1, & \xbfm \in \Gamma_0, \\ -1, & \xbfm \in \Gamma_1, \end{cases}
\end{equation*}
and $\Gamma_D \isdef \Gamma_2$.

The respective visualisations of the reference solution's mean and variance are depicted in 
Figure~\ref{fig:exp2mv}.
 
\begin{figure}[htb]
\centering
\begin{tikzpicture}[baseline]
  \begin{axis}[
      name=plot1,
      title={Mean},
      xlabel={$x_3$},
      ylabel={$x_1$},
      zlabel={$x_2$},
            zlabel style = {rotate=-90},
      width=6cm,
      height=6cm,
      scale mode=scale uniformly,
    ]
    \addplot3 graphics[
      points={
        (1,0,0) => (0,149)
        (0,1,0) => (1577,258)
        (0,0,1) => (577,1523)
        (1,1,1) => (1004,1119)
        (1,1,0)
        (1,0,1)
        (0,1,1)
        (0,0,0)
    }] {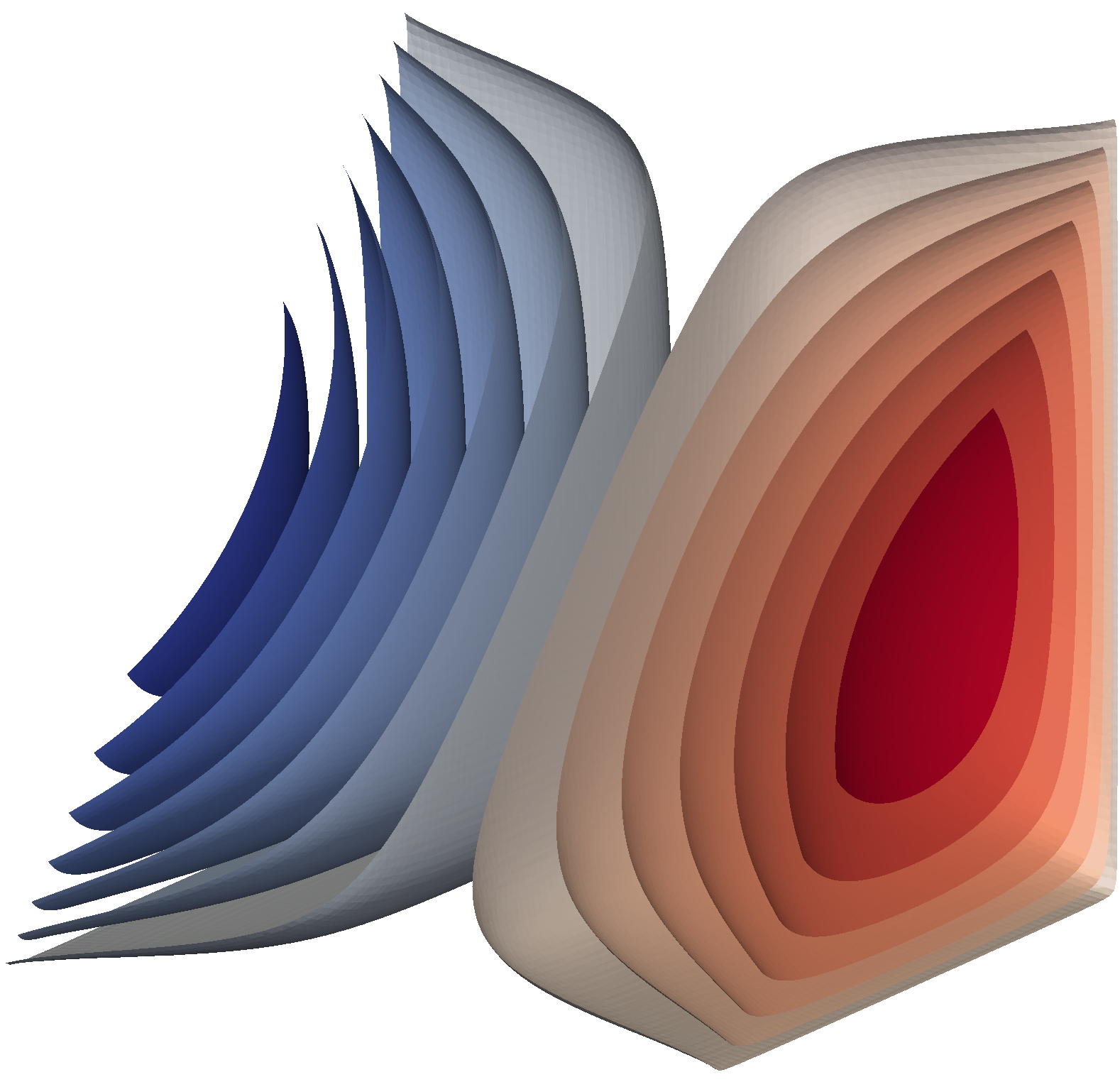};
  \end{axis}
  \begin{axis}[
      name=plot2,
      at=(plot1.below south east), anchor=above north east,
      yshift=-2ex,
      enlargelimits=false,
      axis on top,
      ytick=\empty,
      width=6cm,
      height=2cm,
    ] \addplot graphics[
      ymin=0, ymax=0.1,
      xmin=-0.527455, xmax=0.527295,
    ] {./colorbar.png};
  \end{axis}
\end{tikzpicture}%
\hspace*{0.6cm}%
\begin{tikzpicture}[baseline]
  \begin{axis}[
      name=plot1,
      title={Variance},
      xlabel={$x_3$},
      ylabel={$x_1$},
      zlabel={$x_2$},
            zlabel style = {rotate=-90},
      zticklabel pos=right,
      width=6cm,
      height=6cm,
      scale mode=scale uniformly,
    ]
    \addplot3 graphics[
      points={
        (1,0,0) => (0,149)
        (0,1,0) => (1577,258)
        (0,0,1) => (577,1523)
        (1,1,1) => (1004,1119)
        (1,1,0)
        (1,0,1)
        (0,1,1)
        (0,0,0)
    }] {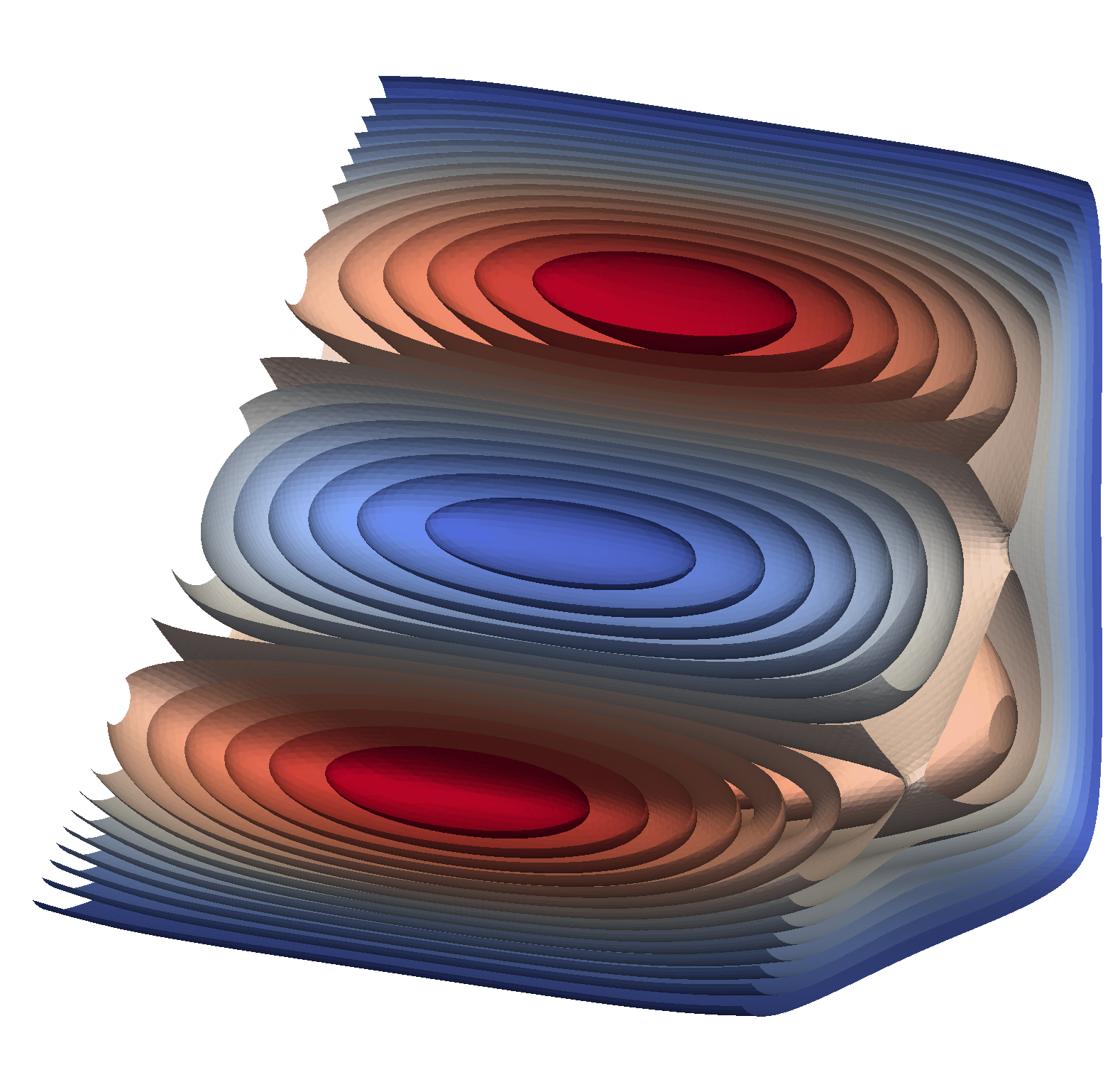};
  \end{axis}
  \begin{axis}[
      name=plot2,
      at=(plot1.below south east), anchor=above north east,
      yshift=-2ex,
      enlargelimits=false,
      axis on top,
      ytick=\empty,
      width=6cm,
      height=2cm,
    ] \addplot graphics[
      ymin=0, ymax=0.1,
      xmin=0.0, xmax=0.0175915,
    ] {./colorbar.png};
  \end{axis}
\end{tikzpicture}
\caption{Mean and variance of the solution.}
\label{fig:exp2mv}
\end{figure}

Figures~\ref{fig:exp2errormean} and~\ref{fig:exp2errorvar} exhibit the estimated errors of the solution's 
mean on the left hand side and of the solution's variance on the right hand 
side, each versus the discretisation level for the different quadrature methods.
Again, each of the quadrature methods achieves the predicted rate of convergence. As in the previous
example QMC and SG provide slightly better errors in case of the variance.

\begin{figure}[htb]
\centering
\begin{minipage}{0.49\textwidth}
\centering
\begin{tikzpicture}[baseline]
  \begin{semilogyaxis}[
      xlabel={Level $l$},
      ylabel={Estimated error},
      xmajorgrids=true,
      yminorgrids=true,
      ymin=3e-2,
    ]
    \addplot[red, mark=o, mark size=1mm,line width=1pt] table [x=l,y=mH1] {./exp2errQMC.dat};
    \addlegendentry{QMC}
    \addplot[blue, mark=square, mark size=1mm,line width=1pt] table [x=l,y=mH1] {./exp2errSQ.dat};
    \addlegendentry{SG}
    \addplot[green!60!black, mark=diamond, mark size=1mm,line width=1pt] table [x=l,y=mH1] {./exp2errMC.dat};
    \addlegendentry{MC}
    \addplot[black, dashed] table [x=l,y=mH1a1] {./exp2errMC.dat};
    \addlegendentry{slope $2^{-l}$}
  \end{semilogyaxis}
\end{tikzpicture}
\caption{\(H^1\)-error in the mean.}
\label{fig:exp2errormean}
\end{minipage}
\begin{minipage}{0.49\textwidth}
\centering
\begin{tikzpicture}[baseline]
  \begin{semilogyaxis}[
      xlabel={Level $l$},
      ylabel={Estimated error},
      xmajorgrids=true,
      yminorgrids=true,
            yticklabel pos=right,
    ]
    \addplot[red, mark=o, mark size=1mm,line width=1pt] table [x=l,y=vW11] {./exp2errQMC.dat};
    \addlegendentry{QMC}
    \addplot[blue, mark=square, mark size=1mm,line width=1pt] table [x=l,y=vW11] {./exp2errSQ.dat};
    \addlegendentry{SG}
    \addplot[green!60!black, mark=diamond, mark size=1mm,line width=1pt] table [x=l,y=vW11] {./exp2errMC.dat};
    \addlegendentry{MC}
    \addplot[black, dashed] table [x=l,y=vW11a1] {./exp2errMC.dat};
    \addlegendentry{slope $2^{-l}$}
  \end{semilogyaxis}
\end{tikzpicture}
\caption{\(W^{1,1}\)-error in the variance.}
\label{fig:exp2errorvar}
\end{minipage}
\end{figure}

\begin{figure}[htb]
\centering
\begin{tikzpicture}[baseline]
  \begin{axis}[
      title={$\Vbfm^{(0)}$},
      xlabel={$x_3$},
      ylabel={$x_1$},
      zlabel={$x_2$},
            zlabel style = {rotate=-90},
      width=6cm,
      height=6cm,
      scale mode=scale uniformly,
    ]
    \addplot3 graphics[
      points={
        (1,0,0) => (0,149)
        (0,1,0) => (1577,258)
        (0,0,1) => (577,1523)
        (1,1,1) => (1004,1119)
        (1,1,0)
        (1,0,1)
        (0,1,1)
        (0,0,0)
    }] {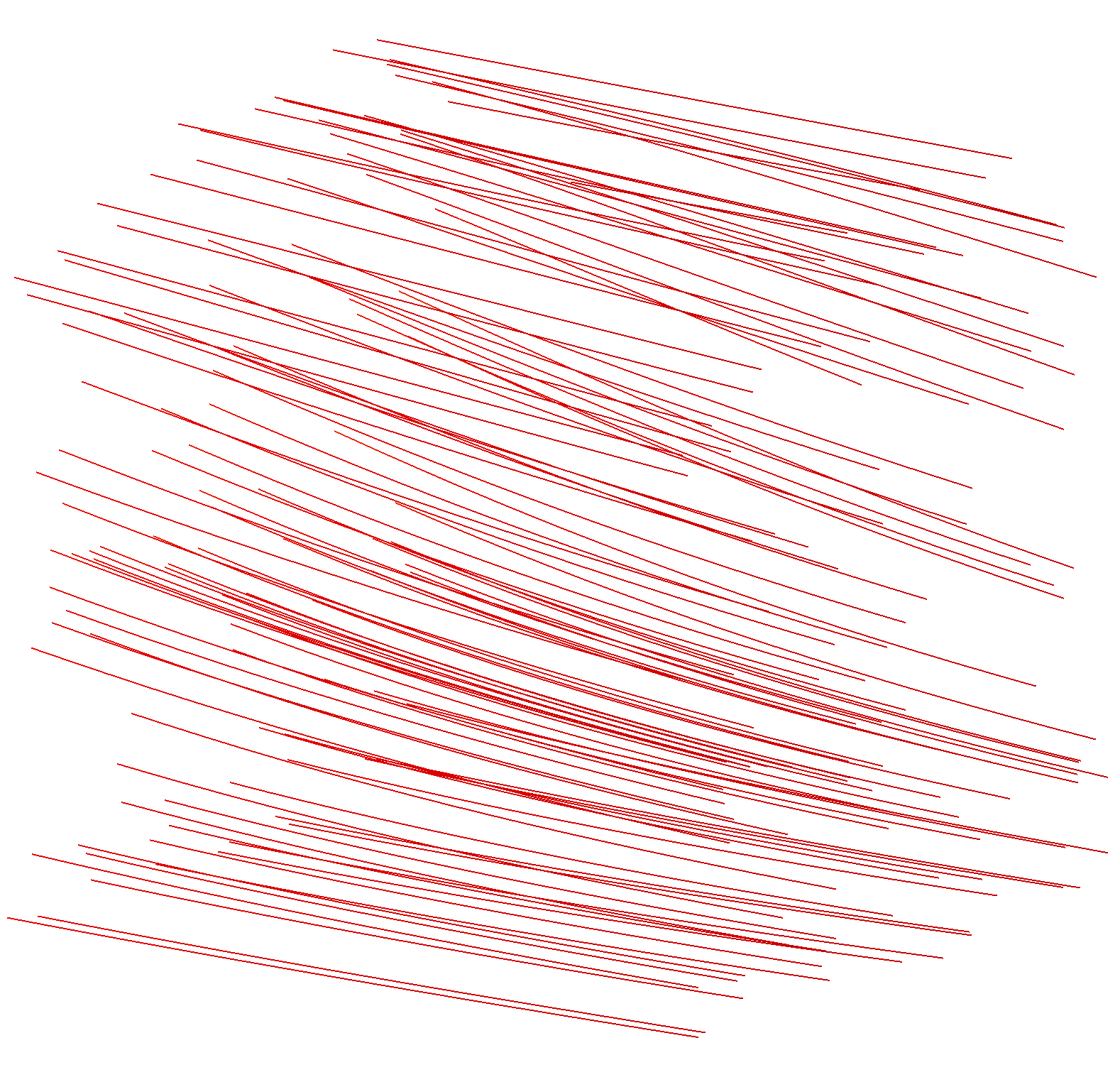};
  \end{axis}
\end{tikzpicture}%
\hspace*{0.6cm}%
\begin{tikzpicture}[baseline]
  \begin{axis}[
      title={$\Vbfm^{(1)}$},
      xlabel={$x_3$},
      ylabel={$x_1$},
      zlabel={$x_2$},
            zlabel style = {rotate=-90},
      zticklabel pos=right,
      width=6cm,
      height=6cm,
      scale mode=scale uniformly,
    ]
    \addplot3 graphics[
      points={
        (1,0,0) => (0,149)
        (0,1,0) => (1577,258)
        (0,0,1) => (577,1523)
        (1,1,1) => (1004,1119)
        (1,1,0)
        (1,0,1)
        (0,1,1)
        (0,0,0)
    }] {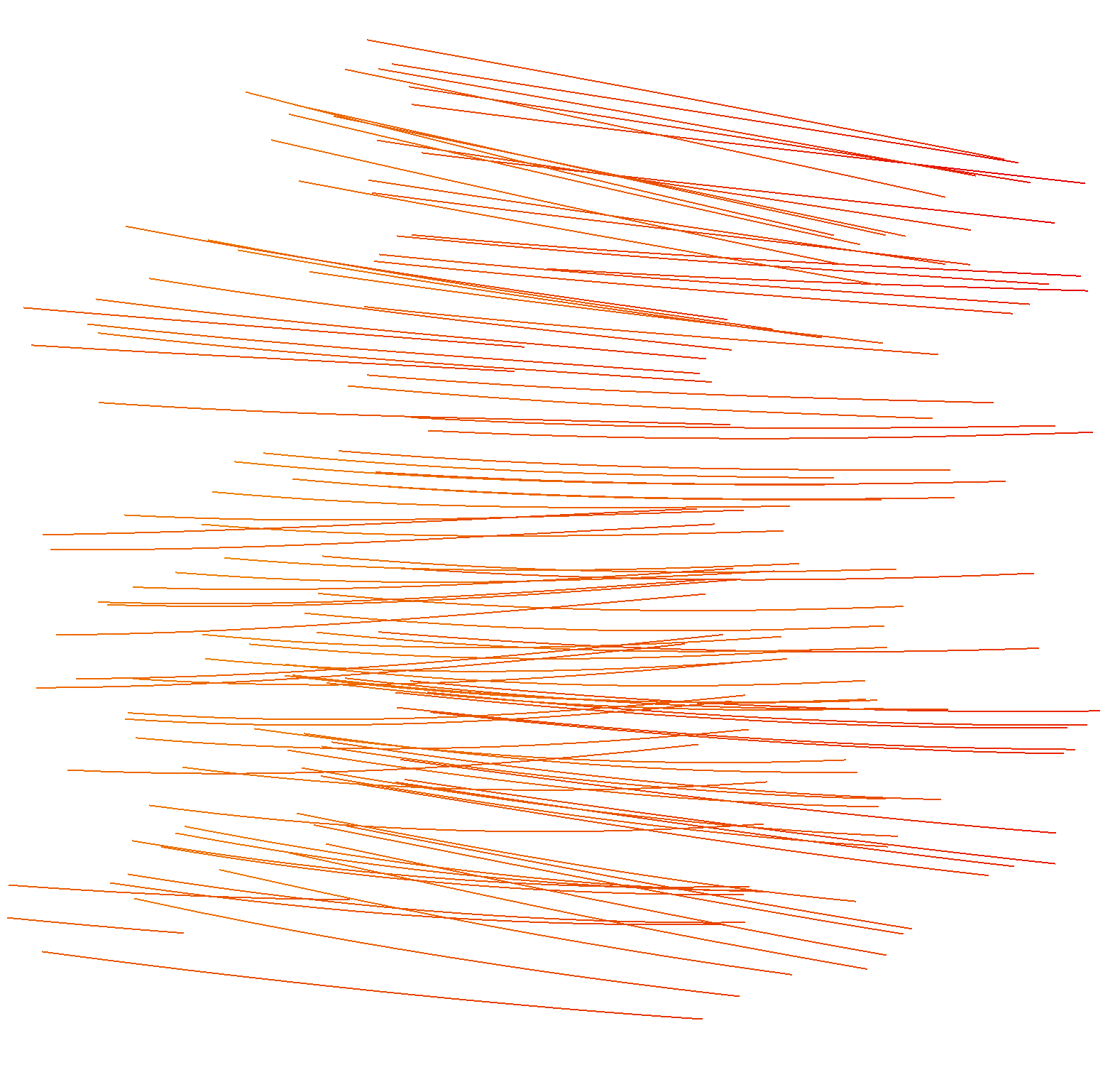};
  \end{axis}
\end{tikzpicture}\\
\begin{tikzpicture}[baseline]
  \begin{axis}[
      title={$\Vbfm^{(2)}$},
      xlabel={$x_3$},
      ylabel={$x_1$},
      zlabel={$x_2$},
            zlabel style = {rotate=-90},
      width=6cm,
      height=6cm,
      scale mode=scale uniformly,
    ]
    \addplot3 graphics[
      points={
        (1,0,0) => (0,149)
        (0,1,0) => (1577,258)
        (0,0,1) => (577,1523)
        (1,1,1) => (1004,1119)
        (1,1,0)
        (1,0,1)
        (0,1,1)
        (0,0,0)
    }] {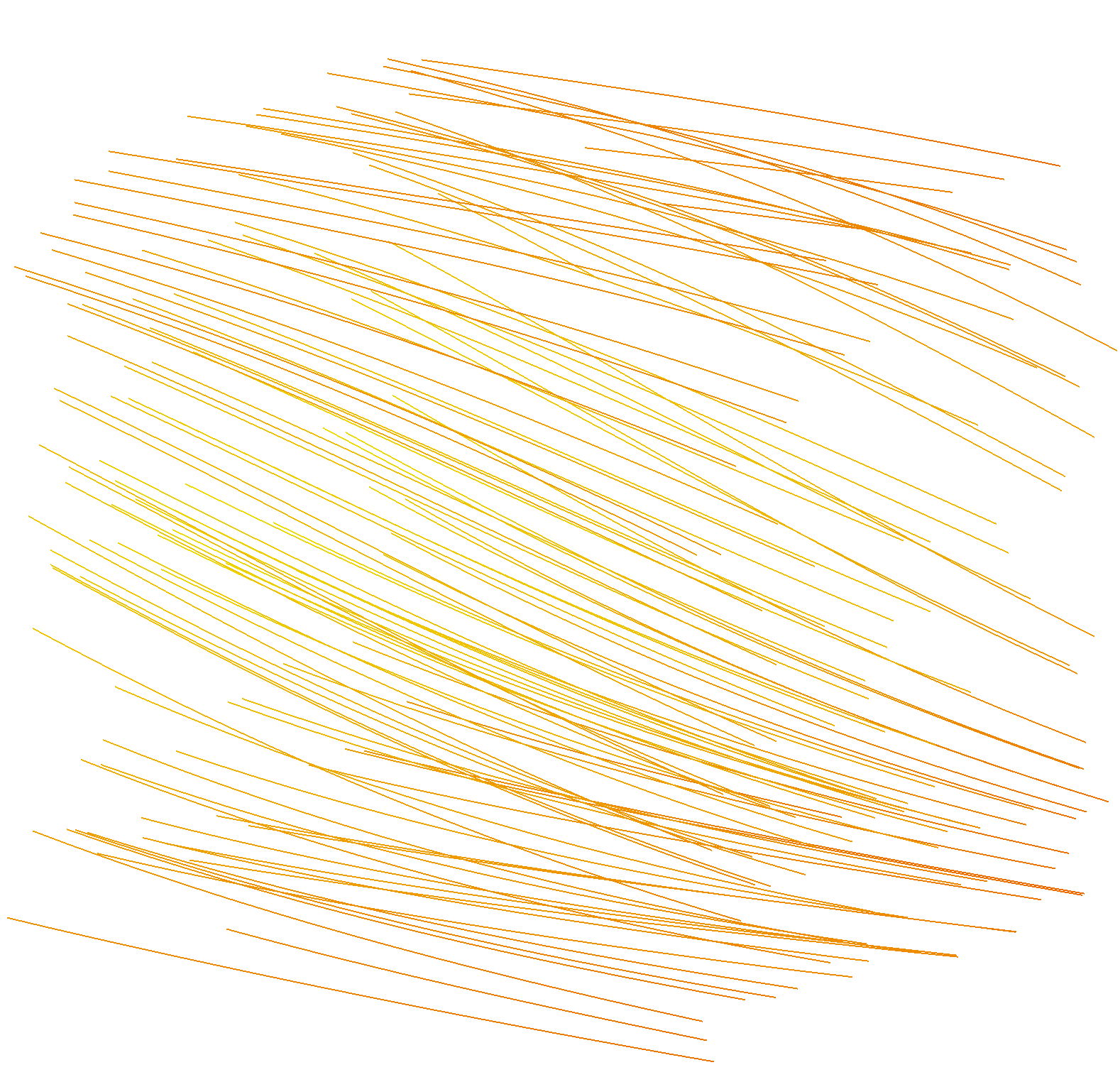};
  \end{axis}
\end{tikzpicture}%
\hspace*{0.6cm}%
\begin{tikzpicture}[baseline]
  \begin{axis}[
      title={$\Vbfm^{(3)}$},
      xlabel={$x_3$},
      ylabel={$x_1$},
      zlabel={$x_2$},
            zlabel style = {rotate=-90},
      zticklabel pos=right,
      width=6cm,
      height=6cm,
      scale mode=scale uniformly,
    ]
    \addplot3 graphics[
      points={
        (1,0,0) => (0,149)
        (0,1,0) => (1577,258)
        (0,0,1) => (577,1523)
        (1,1,1) => (1004,1119)
        (1,1,0)
        (1,0,1)
        (0,1,1)
        (0,0,0)
    }] {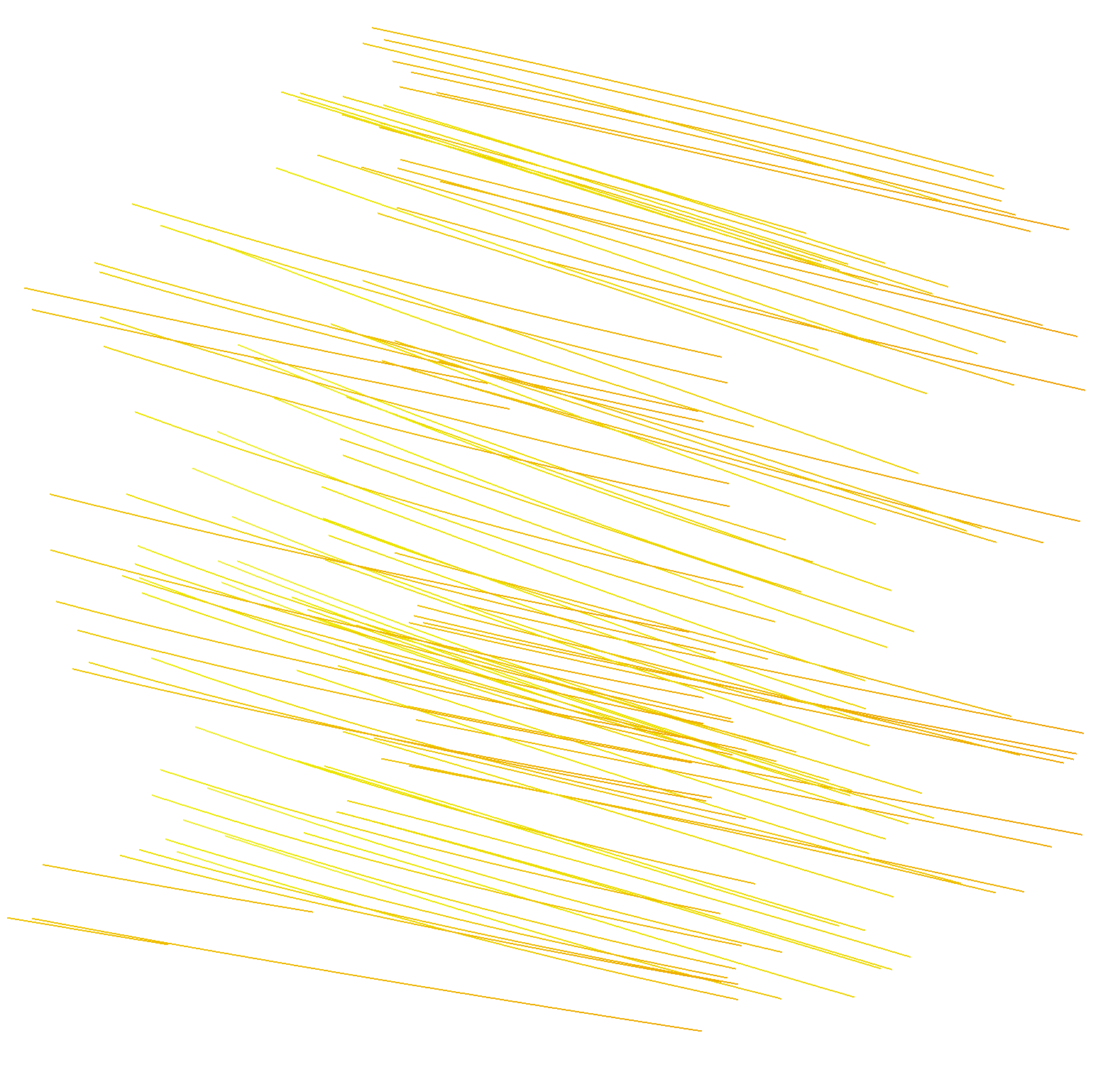};
  \end{axis}
\end{tikzpicture}\\
\begin{tikzpicture}[baseline]
  \begin{axis}[
      title={$\Vbfm^{(4)}$},
      xlabel={$x_3$},
      ylabel={$x_1$},
      zlabel={$x_2$},
            zlabel style = {rotate=-90},
      width=6cm,
      height=6cm,
      scale mode=scale uniformly,
    ]
    \addplot3 graphics[
      points={
        (1,0,0) => (0,149)
        (0,1,0) => (1577,258)
        (0,0,1) => (577,1523)
        (1,1,1) => (1004,1119)
        (1,1,0)
        (1,0,1)
        (0,1,1)
        (0,0,0)
    }] {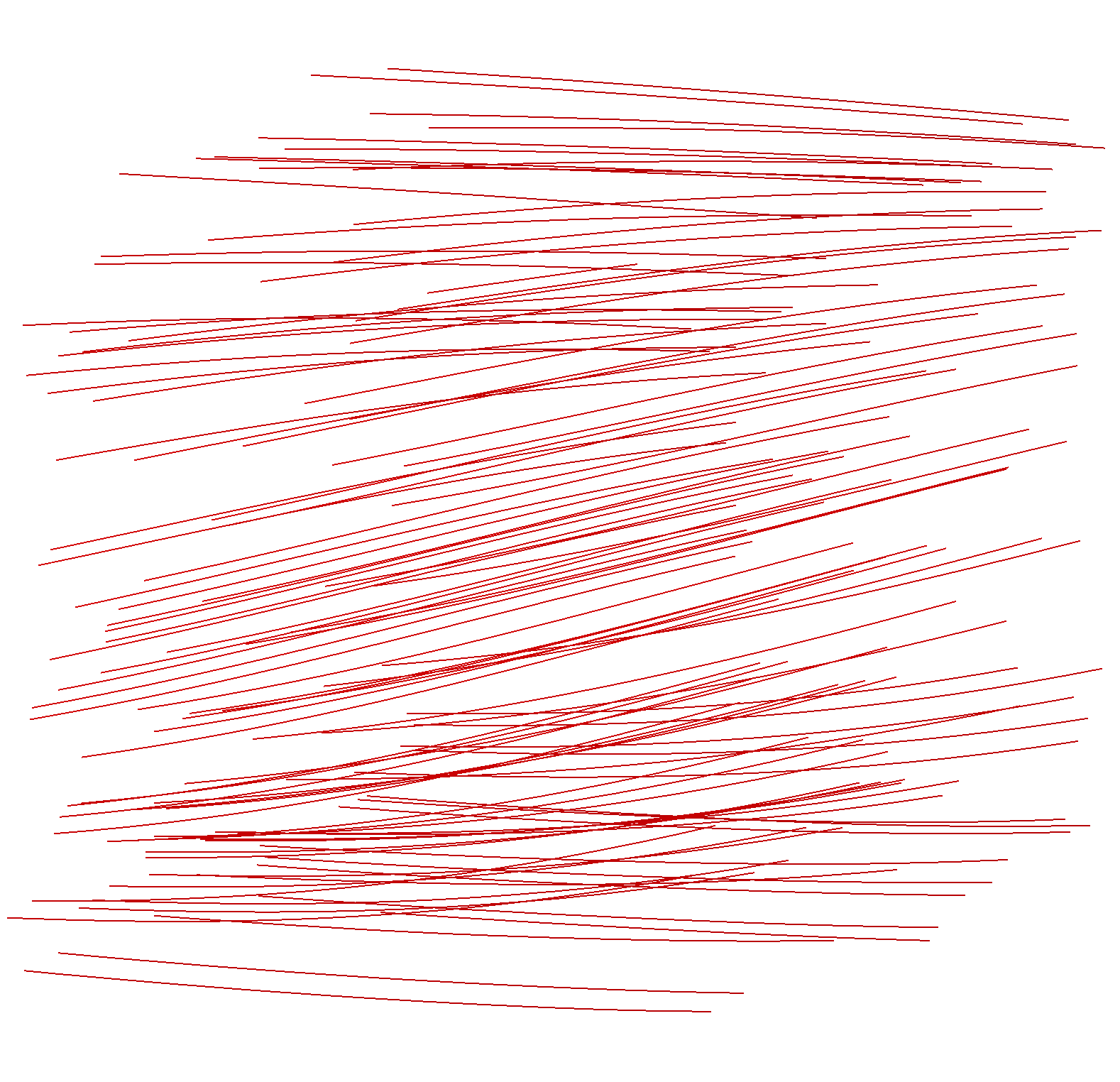};
  \end{axis}
\end{tikzpicture}%
\hspace*{0.6cm}%
\begin{tikzpicture}[baseline]
  \begin{axis}[
      title={$\Vbfm^{(5)}$},
      xlabel={$x_3$},
      ylabel={$x_1$},
      zlabel={$x_2$},
            zlabel style = {rotate=-90},
      zticklabel pos=right,
      width=6cm,
      height=6cm,
      scale mode=scale uniformly,
    ]
    \addplot3 graphics[
      points={
        (1,0,0) => (0,149)
        (0,1,0) => (1577,258)
        (0,0,1) => (577,1523)
        (1,1,1) => (1004,1119)
        (1,1,0)
        (1,0,1)
        (0,1,1)
        (0,0,0)
    }] {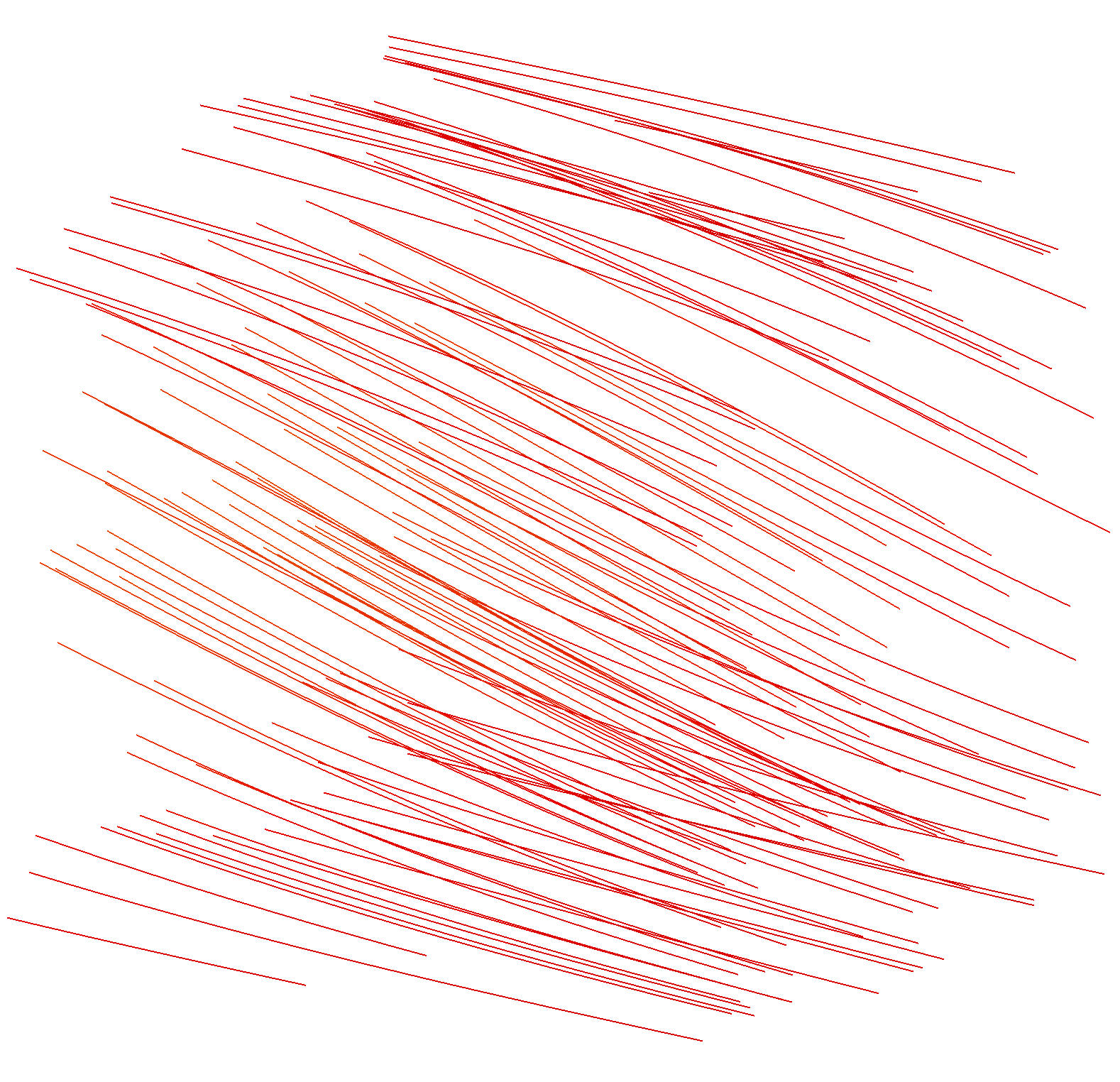};
  \end{axis}
\end{tikzpicture}
\caption{Stream traces of some samples of the vector field $\Vbfm/\|\Vbfm\|_2$ used for the examples.}
\label{fig:exp12V}
\end{figure}

\section{Conclusion}
In this article, we have introduced the diffusion coefficient
\eqref{eq:AdmV} that may be used to model anisotropic diffusion 
that has a notable direction of diffusion with an associated strength,
which both are considered to be subject to uncertainty;
this is encoded by the vector field $\Vbfm$. While this type 
of diffusion coefficient does not model all possible anisotropic 
diffusion coefficients, it can be used to model both diffusion in 
media that consist of thin fibres or thin sheets, given that either 
the diffusion between the fibres or in the sheets is isotropic 
with a global strength that is not subject to uncertainty.

We derived, based on the decay of the Karhunen\hyp{}Lo\`eve 
expansion of $\Vbfm$, related decay rates for the solution's derivatives;
given a sufficiently fast decaying Karhunen\hyp{}Lo\`eve expansion,
this regularity then provides dimension independent convergence
when considering the quasi\hyp{}Monte Carlo quadrature to
approximate quantities of interest that require the integration
of the solution with respect to the random parameter. Furthermore, 
it also allows the use of other quadrature methods like the anisotropic 
sparse grid quadrature which has been considered in the numerical
experiments. The numerical results corroborate the theoretical findings.


Lastly, note that the model for the diffusion coefficient
may, for example, be generalised to
\begin{equation*}
  \Abfm(\xbfm, \omega) \isdef a(\xbfm, \omega)\Ibfm + \groupp[\Big]{\norm[\big]{\Vbfm(\xbfm, \omega)}_2 - a(\xbfm, \omega)} \frac{\Vbfm(\xbfm, \omega) \Vbfm^\trans(\xbfm, \omega)}{\Vbfm^\trans(\xbfm, \omega) \Vbfm(\xbfm, \omega)} ,
\end{equation*}
without effecting regularity.
With this type of diffusion coefficient, the diffusion accounted for
between the fibres or in the sheets is still isotropic,
but can be spatially varying and also subject to uncertainty.
We especially note that for two spatial dimensions this
generalisation already models all types of random anisotropic diffusion coefficients.

\bibliographystyle{plain}
\bibliography{biblio}
\end{document}